\documentclass{amsart}
\usepackage{amssymb,amsmath}
\usepackage{amsfonts}
\usepackage{pstricks, pst-node}
\usepackage[all]{xy}
\numberwithin{equation}{section}
\theoremstyle{plain}
\newtheorem{lemma}{Lemma}

\begin{document}
\title{Exceptional Spin Groups on Hyperelliptic Riemann Surfaces}
\author{K.~M.Bugajska}
\address{Department of Mathematics ans Statistics,
York University,
Toronto,ON, M3J 1P3,
Canada}
\begin{abstract}
We find all exceptional spin groups  attached to the vertices of any exceptional spin-graph on any hyperbolic Riemann surface $\Sigma$ of genus $g\geq{2}$.  In particular we show that when the order $r$ of a graph is $r\geq{3}$ (i.e. the genus of $\Sigma$ must be $g\geq{4}$) then the  spin group attached to an exceptional point $Q$ is either isomorphic to the symmetry group $\mathtt{S}_{r}$ (when ${deg_{\epsilon}Q}=r$) or to the symmetry group $\mathtt{S}_{r+1}$ (when ${deg_{\epsilon}Q}={r+1}$).
\end{abstract}

\maketitle

\section{INTRODUCTION}
In this paper we will consider only a hyperelliptic Riemann surface $\Sigma$ of any genus $g>1$. Let $\xi_{\epsilon}$ be an even, nonsingular spin bundle on  $\Sigma$.  For any point $P\in{\Sigma}$ there exists unique meromorphic section $\sigma_{P}$ of $\xi_{\epsilon}$ with a single simple pole at $P$.  The zero divisor of $\sigma_{P}$ is an integral divisor $\mathcal{A}^{\epsilon}_{P}$ of degree $g$ with vanishing index of specialty ~\cite{RM95}, ~\cite{DV11}, ~\cite{RG67}.

It was shown  ~\cite{KB13} that the spin bundle $\xi_{\epsilon}$ determines an $\epsilon$-foliation of $\Sigma$ whose each leaf carries an additional structure of a spin graph . Almost all leaves of this foliation have $2g+2$ points  that form the vertices of a standard graph.  Besides standard leaves there are two leaves through Weierstrass points (corresponding to two Weierstrass graphs with $g+1$ vertices each) and, necessary, some finite number of exceptional leaves.

Let $P$ be a standard or a Weierstrass point of $\Sigma$ and let $\mathcal{S}^{\epsilon}_{P}$ denote the spin graph through $P$. In this case the divisor $\mathcal{A}^{\epsilon}_{P}$ consists of $g$ distinct points, say, ${\mathcal{A}^{\epsilon}_{P}}=P_{1}P_{2}{\ldots}P_{g}$ and ${\widetilde{P}}\neq{P_{l}}$ for any $l=1,2,\ldots,g$.  It is important that any concrete enumeration of the points of the set $\{\mathcal{A}^{\epsilon}_{P}\}$ determines unique enumeration of the points of $\{\mathcal{A}^{\epsilon}_{Q}\}$ for any other vertex $Q$ of the spin graph $\mathcal{S}^{\epsilon}_{P}$.

It was shown  in ~\cite{KMB13} that a spin-graph structure of  any non-exceptional leaf of the $\epsilon$-foliation implies that at any non-exceptional  point $P\in{\Sigma}$ we may attach the spin group $\mathtt{G}_{P}$.  Each such group is finite  and it is isomorphic to the alternating group of permutations $\mathtt{A}_{g}\triangleleft{\mathtt{S}_{g}}$  acting on the set $\widehat{P}={\Phi_{P}(\Sigma)}\cap{\Theta_{\epsilon}}$. (Here $\Phi_{P}: {\Sigma}\rightarrow{Jac\Sigma}$ is a Jacobi mapping originatingat $P$  and ${\Theta_{\epsilon}}\subset{Jac\Sigma}$ is the appropriate theta divisor   ~\cite{FK92},  ~\cite{FK01}.)  Since there is a  natural identification between the  (ordered) set $\widehat{P}$ and  the (ordered) set $\{\mathcal{A}^{\epsilon}_{P}\}\cong{\{1,\ldots,g\}_{P}}$ we may view  $\mathtt{G}_{P}$  as the group acting on the set of $g$ points  that form the integral divisor $\mathcal{A}^{\epsilon}_{P}$.

 Any hyperelliptic Riemann surface $\Sigma$ of genus $g>1$   must have exceptional points that form the exceptional leaves of the $\epsilon$-foliation. The number of such points is finite and it is smaller than or equal to $4g$. 
In this paper we will   find the spin groups that are attached to such exceptional points. 

For any exceptional spin graph $\mathcal{S}^{\epsilon}_{P}$ we will introduce the notion of the basic graph $\mathcal{S}(r)\subset{\mathcal{S}^{\epsilon}_{P}}$, $r<g$, and the notion of the connection graph for $\mathcal{S}^{\epsilon}_{P}$.  The basic graph $\mathcal{S}(r)$ is isomorphic to the standard graph on a surface of genus equal to $r$ and an exceptional spin graph $\mathcal{S}^{\epsilon}_{P}$ can be obtained  by decorating its basic graph. The connection graph is obtained from $\mathcal{S}(r)$ by adding edges between all conjugate vertices that are connected in $\mathcal{S}^{\epsilon}_{P}$.

It occurs that to find spin groups that are attached to vertices of an exceptional spin graph it is enough to consider only its connection graph.

 A standard or  a Weierstrass graphs are totally symmetric with respect to all of its vertices,  i.e. the properties of all vertices of a graph are exactly the same. It is seldom true for exceptional graphs. However, it occurs for any genus, for example, when the genus of a surface $\Sigma$  is equivalent to $2$ modulo $3$ and when a graph belongs to the isomorphic class $\mathcal{S}_{i_{\max},0,0}$.  However, more often the connection graphs of spin graphs are totally symmetric with respect to all of its vertices. For this it is only required that all vertices of an exceptional spin graph have the same $\epsilon$-degrees.

For  any vertex $Q$ of an exceptional spin graph $\mathcal{S}^{\epsilon}_{P}$
    we will introduce the notion of a spin chain $\mathsf{W}_{Q}$  at $Q$  (similarly as in a standard  and in a Weierstrass case).  
 Its definition will be the same as in a standard case only when the order $r$ of the graph is $r\leq{3}$. This means that when $r\leq{3}$ a chain $\mathsf{W}_{Q}$ will be given by a loop $\mathsf{L}_{Q}$ at $Q$ together with a choice of faces along its edges.

When $r>3$ then the definition of a spin chain $\mathsf{W}_{Q}$ will involve also a choice of some standard $3$-cells  (see definition $10$).

Moreover, in the contrary to a standard or to a Weierstrass case, not all spin chains will be  $\epsilon$-admissible. When $\mathsf{W}_{Q}$ is not admissible, i.e.  when compositions of the mappings determined by this chain are not defined, then we assume that such chain does not move elements of the set $\widehat{Q}$ at all. On the other hand, any admissible chain $\mathsf{W}_{Q}$ determines an isomorphism of $\widehat{Q}$ which can be represented as a permutation of this set.

When all vertices of the loop  $\mathsf{L}_{Q}$ of a chain $\mathsf{W}_{Q}$ have the same $\epsilon$-degrees then this is enough to guarantee that this chain is admissible.  However,  at any $Q\in{\{\mathcal{S}^{\epsilon}_{P}\}}$   there are plenty of admissible chains  which have vertices  with different $\epsilon$-degrees. 

A case of an exceptional graph of order $r=3$ will be investigate very carefully. The reason for this is that for an arbitrary order $r>3$ any decorated $3$-cell of an exceptional graph $\mathcal{S}_{P}$ is either isomorphic to the standard spin graph $\mathcal{S}(3)$ or to some of the exceptional spin graph of order $3$.

This paper is organized as follows:   Section $2$  contains preliminaries . In section $3$ we consider exceptional graphs of order $r\leq{2}$. It is shown that when the order is $r<2$ then all exceptional spin group must be trivial. When the order of a graph is $r=2$ (it is possible only on a surface of genus $g\geq{3}$) then the spin group attached to a vertex of such graph can be eiter trivial or isomorphic to the cyclic group $\mathcal{C}_{3}$ or isomorphic to the symmetry group $\mathtt{S}_{2}$ or to $\mathtt{S}_{3}$.

In section $4$ we investigate exceptional graphs of order $r=3$.  We find that the spin group $\mathtt{G}_{Q}$ attached to a vertex $Q$  of such graph is either isomorphic to the symmetry group $\mathtt{S}_{3}$ (when $deg_{\epsilon}Q=3$) or to the symmetry group $\mathtt{S}_{4}$ (when $deg_{\epsilon}Q=4$).

We consider exceptional spin graphs of any order $r>3$ in section $5$. We use decorated $3$-cells of these graphs and the results of the section $4$ to show that for any vertex $Q$ the spin group $\mathtt{G}_{Q}$ must be isomorphic either to the group $\mathtt{S}_{r}$ (when $deg_{\epsilon}Q=r$ ) or to the group $\mathtt{S}_{r+1}$ (when $deg_{\epsilon}Q={r+1}$). 

Now,  almost all points of $\Sigma$  are standard points or the Weierstrass points and hence we may   classify (see \cite{KB13}) hyperelliptic Riemann surfaces by the exceptional spin graphs that are present on a surface. (The number of such graphs 
must satisfy a concrete condition ~\cite{KB13}.) The results of this paper , together with the fact that  (see ~\cite{KMB13}) the spin groups at almost all  (standard or Weierstrass) points of $\Sigma$ are isomorphic to the alternating group ${\mathtt{A}_{g}}\triangleleft{\mathtt{S_{g}}}$, $g>r$, allow us to classify the hyperelliptic Riemann surfaces using  the concrete number of concrete exceptional spin groups that are attached to exceptional points.

\section{PRELIMINARIES}

Let $\mathcal{S}(r)$ denotes the isomorphic class of  standard spin graphs on a surface of genus $r$. It has $2r+2$ vertices and no pair of conjugate vertices are connected by an edge.   Such graph is totally symmetric with respect to all of its vertices.
\newtheorem{definition}{Definition}
\begin{definition}
Let $\mathcal{S}^{\epsilon}_{P}$ be an exceptional spin graph on a surface $\Sigma$ whose  genus $g$ is $g>2$. The graph $\mathcal{S}(r)\subset{\mathcal{S}^{\epsilon}_{P}}$ with the greatest possible $r\geq{2}$  will be called the basic spin graph of $\mathcal{S}^{\epsilon}_{P}$.
\end{definition}
  Obviously we  have $r<g$.  Now, any exceptional graph can be obtained from its basic graph by the following operations;
\begin{enumerate}
\item Replace a simple straight edge by a multiple straight edge.
\item Add an arc edge with multiplicity $l\geq{1}$ to a straight edge.
\item Add an edge between some conjugate vertices (always straight with appropriate multiplicity).
\end{enumerate}
 We may say that an exceptional spin graph $\mathcal{S}^{\epsilon}_{P}$ is obtained from its basic graph $\mathcal{S}(r)$, $r<g$, by  decorating  the latter one. In the future the subscript '$\epsilon$'   will be usually omitted. 

 Any exceptional spin graph $\mathcal{S}_{P}$ whose base graph is $\mathcal{S}(r)$, $2<r<g$, has $2r+2$ vertices.  For any vertex $Q\in{\{\mathcal{S}_{P}\}}$  the divisor  $\mathcal{A}_{Q}$ must have the form  $\mathcal{A}_{Q}={\widetilde{Q}^{i}Q^{k_{1}}_{1}Q_{2}^{k_{2}}{\ldots}Q_{r}^{k_{r}}}$.

    We  have $k_{l}\geq{1}$   for each $l={1,2,..,r}$,   the exponent $i\geq{0}$,  and $i+k_{1}+k_{2}+{\ldots}+k_{r}=g$. When $r\leq{1}$ then the corresponding exceptional spin graph does not have a basic graph.
Let $k_{0}=i+1$ and let ${\widehat{k}}(Q)$ denote the $(r+1)$-tuple given by $\widehat{k}={\widehat{k}(Q)}=(k_{0},k_{1},\ldots,k_{r})$.  
\begin{definition}
A vertex $Q$ of a spin graph $\mathcal{S}_{P}$ will be called a head of the graph if the multiplicity $i={k_{0}-1}$ of the point $\widetilde{Q}$ in the integral divisor $\mathcal{A}_{Q}$ is the smallest one. Equivalently, when $k_{0}(Q)={\min\{k_{0}(R); R\in{\{\mathcal{S}_{P}\}}\}}$.
\end{definition}

Obviously, when a graph is a standard one  or when it is a Weierstrass spin graph then each  vertex is a head with $k_{0}=1=i+1$.

\begin{definition}
Let $P$ be a head of a spin graph $\mathcal{S}_{P}$. The cardinality $r$ of the set of different points of the integral divisor $\mathcal{A}_{P}$ that are distinct from $\widetilde{P}$ will be called the order of the graph $\mathcal{S}_{P}$. 
\end{definition}
Hence, any standard or any Weierstrass graph on $\Sigma$  has its order equal to the genus $g$ of the surface. Any exceptional spin graph has the order   $r$ which satisfies $0\leq{r}<g$. When the order of $\mathcal{S}_{P}$ is $r\geq{2}$ then  this graph has  $\mathcal{S}(r)$ is its basic graph.

When we choose the indexes for the points of the set $\{\mathcal{A}_{P}\}$, i.e. when we write ${\mathcal{A}_{P}}={{\widetilde{P}}^{i}P_{1}^{k_{1}}{\ldots}P_{r}^{k_{r}}}$ then, simultaneously, we determine a unique ordering for each set ${\{\mathcal{A}_{Q}\}}$,  ${Q\in{\{\mathcal{S}_{P}\}}}$.  Equivalently, we obtain a unique ordering of the set ${\widehat{Q}}:={\Phi_{Q}(\{\mathcal{A}_{Q}\})}\subset{\Theta_{\epsilon}}$.

Since the image of a divisor ${\frac{Q}{R}}$ of degree zero under any Jacobi mapping (with arbitrary origin $S\in{\Sigma}$) is exactly the same, we will often denote the point $\Phi_{R}(Q)$ by $\frac{Q}{R}$. 

For any point $P\in{\Sigma}$ its image obtained by the hyperelliptic involution on $\Sigma$ will be denoted by $\widetilde{P}$.  Similarly, for any divisor $D$ on a surface $\Sigma$   by $\widetilde{D}$ we will denote  its image under the hyperelliptic involution.

For any object $\mathcal{O}$  the set of all its points 
 will be denoted by $\{\mathcal{O}\}$.

\section{ SPIN GROUPS ASSOCIATED to GRAPHS of ORDER ${r}\leq{2}$ }
\subsection{$r=0$}
In this case a spin graph has the form shown by the  Pict1.  It is obvious that both spin groups $\mathtt{G}_{P}$ and $\mathtt{G}_{\widetilde{P}}$ are trivial.

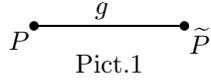
\begin{pspicture}(0,-1)(4,1)
\psline[showpoints=true]%
(1,0)(3,0)
\rput(2,-0.5){\rnode{A}{Pict.1}}
\rput(0.8,-0.2){\rnode{a}{$P$}}
\rput(3.2,-0.2){\rnode{b}{$\widetilde{P}$}}
\rput(1.9,0.2){\rnode{c}{$g$}}

\end{pspicture}

\subsection{$r=1$}.
Let $P$ be a head of an exceptional spin graph $\mathcal{S}_{P}$ of order $r=1$. In this case we have $\mathcal{A}_{P}={\widetilde{P}^{i}P^{k_{1}}_{1}}$ with $i={k_{0}-1}$, $i+k_{1}={g}$ and with $k_{0}\leq{k_{1}={k_{0}+p_{1}}}$, $p_{1}\geq{0}$.   We must have (see[1])
\begin{equation}
 0\leq{i}\leq{\left\lfloor \frac{g-1}{2}\right\rfloor}
 \end{equation}
  Since  ${\widehat{k}(P)}={\widehat{k}}=(k_{0},k_{1})=(i+1,i+1+p_{1})$ we will denote the corresponding class of isomorphic exceptional graphs by $\mathcal{S}_{i,p_{1}}(g)$.  When $i=0$ then $k_{1}=g$ and the class ${\mathcal{S}_{i,p_{1}}}={\mathcal{S}_{0,g-1}}$ is given by Pict2.  
 
 
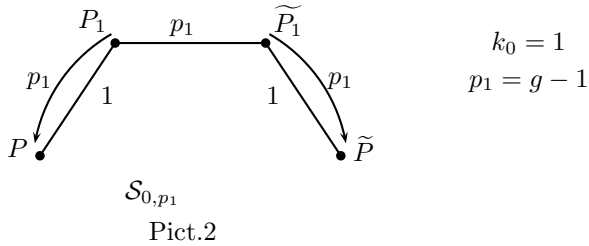
\begin{pspicture}(-2,-2)(6,3)
 \psline[showpoints=true]%
 (-2,0)(-1,1.5)
 \psline[showpoints=true]%
 (1,1.5)(2,0)
 \psline(-1,1.5)(1,1.5)
 \psarc{->}%
 (0,-0.2){2.1}{120}{170}
 \psarc{<-}%
 (0,-0.2){2.1}{10}{60}
 \rput(-0.5,-0.5){\rnode{A}{$\mathcal{S}_{0,p_{1}}$}}
 \rput(-0.1,-1){\rnode{B}{Pict.2}}
 
 \rput(4.5,1.5){\rnode{C}{$k_{0}=1$}}
 \rput(4.5,1){\rnode{D}{$p_{1}={g-1}$}}
 
 \rput(-0.1,1.7){\rnode{a}{$p_{1}$}}
 \rput(-1.3,1.8){\rnode{b}{$P_{1}$}}
 \rput(1.3,1.8){\rnode{c}{$\widetilde{P_{1}}$}}
 \rput(-2,1){\rnode{d}{$p_{1}$}}
 \rput(-1.1,0.8){\rnode{e}{$1$}}
 \rput(1.1,0.8){\rnode{f}{$1$}}
 \rput(2,1){\rnode{g}{$p_{1}$}}
 \rput(-2.3,0.1){\rnode{h}{$P$}}
 \rput(2.3,0.1){\rnode{i}{$\widetilde{P}$}}
 
 \end{pspicture}
 
 When $i\geq{1}$ then all conjugate vertices of a graph $\mathcal{S}_{P}$ are connected and graph belongs to the class $\mathcal{S}_{i,p_{1}}$ given by the Pict3.

 In particular, when the genus of a surface $\Sigma$ is odd then for $i={i_{\max}}={\frac{g-1}{2}}$ we have $p_{1}=0$.  For the remaining values of $i$, i.e. for $i<{\frac{g-1}{2}}$ we have $p_{1}={2i_{\max}-2i}$ and hence it is always even.  When the genus of a surface is an even integer then $i_{\max}={\left\lfloor \frac{g-1}{2}\right\rfloor}<{\frac{g-1}{2}}$ and we must have $p_{1}=1$. For the remaining values of $i$ we have $p_{1}={2i_{\max}-2i+1}\geq{3}$ which is always odd.
 
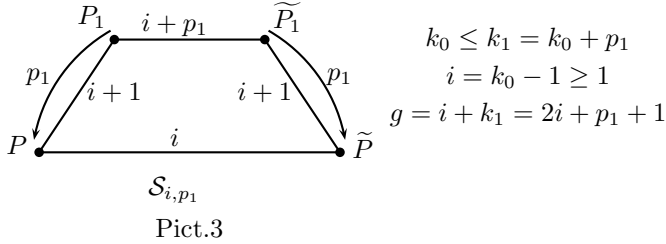
\begin{pspicture}(-2,-2)(6,3)
 \psline[showpoints=true]%
 (-2,0)(-1,1.5)
 \psline[showpoints=true]%
 (1,1.5)(2,0)
 \psline(-1,1.5)(1,1.5)
 \psline(-2,0)(2,0)
 \psarc{->}%
 (0,-0.2){2.1}{120}{170}
 \psarc{<-}%
 (0,-0.2){2.1}{10}{60}
 
 \rput(-0.2,-0.5){\rnode{A}{$\mathcal{S}_{i,p_{1}}$}}
 \rput(0,-1){\rnode{B}{Pict.3}}
 \rput(4.5,1.5){\rnode{C}{$k_{0}\leq{k_{1}}={k_{0}+p_{1}}$}}
 \rput(4.5,1){\rnode{D}{$i={k_{0}-1}\geq{1}$}}
 \rput(4.5,0.5){\rnode{E}{$g={i+k_{1}}={2i+p_{1}+1}$}}
 
 \rput(-0.2,1.7){\rnode{a}{${i+p_{1}}$}}
 \rput(-0.2,0.2){\rnode{b}{$i$}}
 \rput(-2,1){\rnode{c}{$p_{1}$}}
 \rput(2,1){\rnode{d}{$p_{1}$}}
 \rput(-1,0.8){\rnode{e}{${i+1}$}}
 \rput(1,0.8){\rnode{f}{${i+1}$}}
 
 \rput(-2.3,0.1){\rnode{g}{$P$}}
 \rput(2.3,0.1){\rnode{h}{$\widetilde{P}$}}
 \rput(1.3,1.8){\rnode{i}{$\widetilde{P_{1}}$}}
 \rput(-1.3,1.8){\rnode{j}{$P_{1}$}}

 \end{pspicture}
 
 $\mathbf{EXAMPLE}$ $\mathbf{I}$.  Suppose that the genus of $\Sigma$ is $g=2$. Now, any exceptional spin graph of order one on $\Sigma$ belongs to the class $\mathcal{S}_{i,p_{1}}={\mathcal{S}_{0,1}}$  (a particular case of Pict.2).

 $\mathbf{EXAMPLE}$ $\mathbf{II}$.  When $g=3$  then the possible classes $\mathcal{S}_{i,p_{1}}$ of isomorphic spin graphs are $\mathcal{S}_{0,2}$ and $\mathcal{S}_{1,0}$. They are  particular cases of Pict.2 and Pict.3  respectively.

 Let the genus of a  surface $\Sigma$ be $g=4$. Now $\mathcal{S}_{P}$  belongs to a  class $\mathcal{S}_{i,p_{1}}$  with $(i,p_{1})\in{\{(0,3),(1,1)\}}$. Equivalently, if $P$ is a head of such graph then we have $\widehat{k}(P)=(k_{0},k_{1})$ with  $(k_{0},k_{1})\in{\{(1,4),(2,3)\}}$. 
 
  When the genus is $g=5$ then the possible classes  $\mathcal{S}_{i,p_{1}}$ have $(i,p_{1})$ which belong to the set ${\{(0,4),(1,2),(2,0)\}}$ (equivalently, $\widehat{k}(P)=(k_{0},k_{1})$ belongs to the set $\{(1,5),(2,4),(3,3)\}$ respectively).
 
 Since on any surface of genus $g\geq{2}$ any exceptional spin graph of order $r=1$ has at most only one quadrilateral face  it is obvious that the spin goup attached to any of its vertex must be trivial.(see ~\cite{KMB13}).
 
 \subsection{$r=2$}.
 Before we will start to investigate the spin groups attached to exceptional spin graphs of order two let us notice the following: 
 
  For any vertex $Q$ of any graph $\mathcal{S}_{P}$ (on any hyperelliptic Riemann surface) the spin group $\mathtt{G}_{Q}$ is totally determined by the presence of connections (i.e. straight edges) between the vertices of the graph  and not by their multiplicities. Hence, since the multiplicities of edges are irrelevant and since any arc-connection must accompany a straight line connection we observe that to find the exceptional spin group $\mathtt{G}_{Q}$, $Q\in{\{\mathcal{S}_{P}\}}$,  it is enough to consider the basic graph $\mathcal{S}(r)\subset{\mathcal{S}_{P}}$, $r\geq{2}$, $r<g$, together with all additional (straight) edges connecting appropriate conjugate vertices in $\mathcal{S}_{P}$.  
 \begin{definition}
 Let $\mathcal{S}_{P}$ be an exceptional spin graph of order $r\geq{2}$, $r<g$. The basic graph $\mathcal{S}(r)\subset{\mathcal{S}_{P}}$ together with straight edges between conjugate vetrices that are connected in $\mathcal{S}_{P}$ will be called the connection graph for $\mathcal{S}_{P}$. We will draw such connection graps using only dash lines.
 \end{definition}
 Suppose that $P$ is a head of an exceptional spin graph $\mathcal{S}_{P}$ of order $r=2$.  The general form of such graph is given by the Pict4.

 For  $\widehat{k}(Q)=(k_{0},k_{1},k_{2})$,  $Q\in{\{\mathcal{S}_{P}\}}$, we will use the similar notation that before, i.e.   $k_{0}=i+1$, $k_{0}\leq{k_{1}}={k_{0}+p_{1}}\leq{k_{2}}$ and $k_{2}={k_{1}+p_{2}}$.  Since we have $g={i+k_{1}+k_{2}}={3i+2p_{1}+p_{2}+2}$, the maximum possible value   $i$ for a head of the graph is $i_{\max}={\left\lfloor \frac{g-2}{3}\right\rfloor}$.  When $i=i_{\max}$ then, depending on the genus of a surface, we have the following possibilities:
 \begin{itemize}
 \item When $g\equiv{0mod3}$ then $2p_{1}+p_{2}=1$. Hence $p_{1}=0$,  $p_{2}=1$ and the vertices $P$, $P_{1}$ and their conjugates are heads of the graph $\mathcal{S}_{P}$.  
 \item When $g\equiv{1mod3}$ then $2p_{1}+p_{2}=2$ and either $p_{1}=0$, $p_{2}=2$ (and hence $P$, $P_{1}$ and they conjugates are heads) or $p_{1}=1$, $p_{2}=0$ and only $P$ and its conjugate are heads of the graph.
 \item When $g\equiv{2mod3}$ then  $p_{1}=p_{2}=0$ and hence the graph is totally symmetric with respect to all of its vertices.  Each vertex $Q$ of $\mathcal{S}_{P}$ is a head.
 \end{itemize}

 \begin{pspicture}(-4,-4.5)(5.5,4)
 \pspolygon[showpoints=true]%
 (-1.5,-2.6)(1.5,-2.6)(3,0)(1.5,2.6)(-1.5,2.6)(-3,0)
 \rput(5.5,3){\rnode{A}{$k_{0}={i+1}$}}
 \rput(5.5,2.4){\rnode{B}{$k_{1}={k_{0}+p_{1}}$}}
 \rput(5.5,1.8){\rnode{C}{$k_{2}=k_{1}+p_{2}$}}
 \rput(5.5,1.2){\rnode{D}{$g={i+k_{1}+k_{2}}\geq{3}$}}
 \rput(-0.2,-4){\rnode{E}{Pict.4}}
 \rput(5.5,0){\rnode{F}{$s_{1}={i+p_{1}}$}}
 \rput(5.5,-0.6){\rnode{G}{$s_{2}={i+p_{1}+p_{2}}$}}

 \rput(-3.3,0){\rnode{a}{$P$}}
 \rput(3.4,0){\rnode{b}{$\widetilde{P}$}}
 \rput(-1.6,3.1){\rnode{c}{$P_{2}$}}
 \rput(1.6,3.1){\rnode{d}{$\widetilde{P_{1}}$}}
 \rput(-1.6,-3.1){\rnode{e}{$P_{1}$}}
 \rput(1.6,-3.1){\rnode{f}{$\widetilde{P_{2}}$}}

 \psarc{->}%
 (-1.5,0.866){1.8}{95}{205}
 \psarc{<-}%
 (1.5,0.866){1.8}{-25}{85}
 \psarc{->}%
 (0,1.74){1.8}{35}{145}
 \psarc{<-}%
 (-1.5,-0.866){1.8}{155}{265}
 \psarc{->}%
 (1.5,-0.866){1.8}{-85}{25}
 \psarc{->}%
 (0,-1.74){1.8}{-145}{-35}
 
 \psline(-1.5,2.6)(1.5,-2.6)
 \psline(-1.5,-2.6)(-0.1,-0.1)
 \psline(0.1,0.1)(1.5,2.6)
 \psline(-3,0)(-0.1,0)
 \psline(0.1,0)(3,0)
 
 \rput(0,3.3){\rnode{g}{$p_{2}$}}
 \rput(0,2.3){\rnode{h}{$k_{1}$}}
 \rput(0,-2.3){\rnode{i}{$k_{1}$}}
 \rput(0,-3.3){\rnode{j}{$p_{2}$}}
 \rput(-1,1.3){\rnode{k}{$s_{1}$}}
 \rput(1,1.3){\rnode{l}{$s_{2}$}}
 \rput(-2.1,1){\rnode{m}{$k_{0}$}}
 \rput(-2.1,-1){\rnode{n}{$k_{0}$}}
 \rput(-1.5,-0.2){\rnode{o}{$i$}}
 \rput(2.1,1){\rnode{p}{$k_{0}$}}
 \rput(2.1,-1){\rnode{q}{$k_{0}$}}
 \rput(-3,1.4){\rnode{r}{$p_{1}$}}
 \rput(-3.1,-1.4){\rnode{s}{${p_{1}+p_{2}}$}}
 \rput(3.1,1.4){\rnode{t}{${p_{1}+p_{2}}$}}
 \rput(3,-1.4){\rnode{u}{$p_{1}$}}

 \end{pspicture}

Let $P$ be a head of an exceptional spin graph $\mathcal{S}_{P}$.  According to our notation, $\widehat{k}(P)=(k_{0},k_{1},k_{2})=(i+1,i+p_{1}+1,i+p_{1}+p_{2}+1)$ and hence the class of exceptional spin graphs isomorphic to $\mathcal{S}_{P}$ will be denoted by $\mathcal{S}_{i,p_{1},p_{2}}(g)$ (or when the genus is fixed,  by $\mathcal{S}_{i,p_{1},p_{2}}$).

   The connection graph of any exceptional spin graph of degree $r=2$ has one of the forms that are given by Pict.5a, Pict.5b, and by Pict.5c.

 Let us consider the following examples.

$\mathbf{EXAMPLE}$ $\mathbf{I}$. Suppose that the genus of a surface is $g=3$. In this case any exceptional spin graph $\mathcal{S}_{P}$ of order $r=2$ must belong to the class  $\mathcal{S}_{i,p_{1},p_{2}}={\mathcal{S}_{0,0,1}}$. Its connection graph has the form given by Pict.5a.

$\mathbf{EXAMPLE}$ $\mathbf{II}$.  Let genus $g=4$. Now graph $\mathcal{S}_{P}$ may belong either to the class $\mathcal{S}_{0,0,2}$ or to $\mathcal{S}_{0,1,0}$ corresponding to $\widehat{k}(P)=(1,1,3)$ or to $\widehat{k}(P)=(1,2,2)$ respectively. The connection graphs for these exceptional spin graphs   are given by Pict.5a and Pict.5b.

$\mathbf{EXAMPLE}$ $\mathbf{III}$. Suppose that the genus is equal to $5$. Now any exceptional spin graph $\mathcal{S}_{P}$ of degree two must belong to a class $\mathcal{S}_{i,p_{1},p_{2}}$ with $(i,p_{1},p_{2})\in{\{(0,1,1),(0,0,3),(1,0,0)\}}$. Equivalently we have $\widehat{k}(P)$ with
\begin{equation*}
 (k_{0},k_{1},k_{2})\in{\{(1,2,3),(1,1,4),(2,2,2)\}}
 \end{equation*}
  respectively.  In the first two cases the connection graph for $\mathcal{S}_{i,p_{1},p_{2}}$ look like on Pict.5a  and Pict.5b respectively.   In the latter case the graph $\mathcal{S}_{1,0,0}$ is symmetric with respect to all of its vertices.  Its graph is given by Pict.6 and has the connection graph as on Pict.5c.

 \begin{pspicture}(-4,-3)(4,3)
 \pspolygon[showpoints=true,linestyle=dashed]%
 (-3,0)(-2.25,-1.3)(-0.75,-1.3)(0,0)(-0.75,1.3)(-2.25,1.3)
 \psline[linestyle=dashed]%
 (-2.25,1.3)(-0.75,-1.3)
 \rput(-2.25,-1.8){\rnode{A}{$k_{0}=1=k_{1}$;}}
 \rput(-0.75,-1.8){\rnode{B}{$k_{2}\geq{2}$}}
 \rput(-1.5,-2.3){\rnode{C}{Pict.5a}}
 \rput(-2.5,1.4){\rnode{a}{$P_{2}$}}
 \rput(-0.5,1.4){\rnode{b}{$\widetilde{P_{1}}$}}
 \rput(-2.6,-1.1){\rnode{c}{$P_{1}$}}
 \rput(-0.4,-1.1){\rnode{d}{$\widetilde{P_{2}}$}}
 \rput(-3.2,0.2){\rnode{e}{$P$}}
 \rput(0.2,0.2){\rnode{f}{$\widetilde{P}$}}

 \pspolygon[showpoints=true,linestyle=dashed]%
 (2,0)(2.75,-1.3)(4.25,-1.3)(5,0)(4.25,1.3)(2.75,1.3)
 \psline[linestyle=dashed]%
 (4.25,-1.3)(2.75,1.3)
 \psline[linestyle=dashed]%
 (2.75,-1.3)(4.25,1.3)
 \rput(2.75,-1.8){\rnode{D}{$k_{0}=1<{k_{1}}\leq{k_{2}}$}}
 \rput(3.5,-2.3){\rnode{E}{Pict.5b}}
 \rput(2.5,1.4){\rnode{g}{$P_{2}$}}
 \rput(4.5,1.4){\rnode{h}{$\widetilde{P_{1}}$}}
 \rput(1.8,0.2){\rnode{i}{$P$}}
 \rput(5.2,0.2){\rnode{j}{$\widetilde{P}$}}
 \rput(2.4,-1.1){\rnode{k}{$P_{2}$}}
 \rput(4.6,-1.1){\rnode{l}{$\widetilde{P_{2}}$}}

 \end{pspicture}

  Let $Q$ be any vertex of $\mathcal{S}_{P}$.  We observe that its $\epsilon$-degree can be read from the connection graph and that we have  either $deg_{\epsilon}Q=r=2$ (hence $\widehat{Q}\cong{\{1,2\}}_{Q}$) or $deg_{\epsilon}Q={r+1}=3$ (and  hence  $\widehat{Q}\cong{\{1,2,3\}_{Q}}$). (We recall  that, for any vertex $Q$ of $\mathcal{S}_{P}$, the enumeration of elements of $\widehat{Q}={\Phi_{Q}\cap{\Theta_{\epsilon}}}$ is uniquely determined by the enumeration of the points, different than $\widetilde{P}$,
  of the divisor $\mathcal{A}_{P}$.)

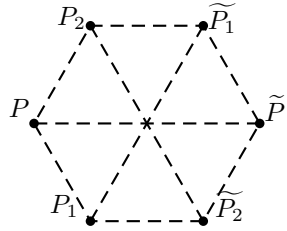
  \begin{pspicture}(-4,-3)(4,3)
  \pspolygon[showpoints=true, linestyle=dashed]%
  (-3,0)(-2.25,-1.3)(-0.75,-1.3)(0,0)(-0.75,1.3)(-2.25,1.3)
  \rput(-1.5,-2.4){\rnode{A}{Pict.5c}}
  \rput(-2.25,-1.8){\rnode{C}{$k_{l}>1;l=0,1,2$}}
  \psline[linestyle=dashed]%
  (-3,0)(0,0)
  \psline[linestyle=dashed]%
  (-2.25,1.3)(-0.75,-1.3)
  \psline[linestyle=dashed]%
  (-2.25,-1.3)(-0.75,1.3)
  \rput(-2.5,1.4){\rnode{a}{$P_{2}$}}
  \rput(-0.5,1.4){\rnode{b}{$\widetilde{P_{1}}$}}
  \rput(-2.6,-1.1){\rnode{c}{$P_{1}$}}
  \rput(-0.4,-1.1){\rnode{d}{$\widetilde{P_{2}}$}}
  \rput(-3.2,0.2){\rnode{e}{$P$}}
  \rput(0.2,0.2){\rnode{f}{$\widetilde{P}$}}

  \pspolygon[showpoints=true]
  (2,0)(2.75,-1.3)(4.25,-1.3)(5,0)(4.25,1.3)(2.75,1.3)
  \rput(2.5,-1.8){\rnode{E}{$g=5;$}}
  \rput(3.5,-1.8){\rnode{D}{$\mathcal{S}_{1,0,0}$}}
  \rput(4,-2.4){\rnode{B}{Pict.6}}
  \psline(2.75,-1.3)(4.25,1.3)
  \psline(2.75,1.3)(3.4,0.2)
  \psline(3.6,-0.2)(4.25,-1.3)
  \psline(2,0)(3.3,0)
  \psline(3.7,0)(5,0)
  \rput(2.5,1.4){\rnode{g}{$P_{2}$}}
  \rput(4.5,1.4){\rnode{h}{$\widetilde{P_{1}}$}}
  \rput(1.8,0.2){\rnode{i}{$P$}}
  \rput(5.2,0.2){\rnode{j}{$\widetilde{P}$}}
  \rput(2.4,-1.1){\rnode{k}{$P_{1}$}}
  \rput(4.6,-1.1){\rnode{l}{$\widetilde{P_{2}}$}}
  
  \rput(3.5,1.5){\rnode{m}{$2$}}
  \rput(3.5,-1.5){\rnode{n}{$2$}}
  \rput(2.2,0.7){\rnode{o}{$2$}}
  \rput(4.8,0.7){\rnode{p}{$2$}}
  \rput(2.1,-0.5){\rnode{q}{$2$}}
  \rput(4.9,-0.5){\rnode{r}{$2$}}
  \rput(3.3,0.7){\rnode{s}{$1$}}
  \rput(2.9,-0.8){\rnode{t}{$1$}}
  \rput(2.5,0.3){\rnode{u}{$1$}}

  \end{pspicture}

  To find the spin group $\mathtt{G}_{Q}$ we intoduce spin chains $\mathsf{W}_{Q}$ at $Q$ which will produce permutations of the set $\widehat{Q}$ resprctively.
  \begin{definition}
  Let $Q$ be a vertex of an exceptional spin graph $\mathcal{S}_{P}$ of order $r\leq{3}$. A spin chain $\mathsf{W}_{Q}$ at $Q$ consists of a loop ${\mathsf{L}_{Q}}=QQ_{1}Q_{2}{\ldots}Q_{N}Q$ in $\mathcal{S}_{P}$ together with a choice of faces $F_{1},F_{2},\ldots,F_{N+1}$ along the edges of this loop. The order of points in $\mathsf{L}_{Q}$ indicate the direction of traveling along this loop. The edge $QQ_{1}$ is an edge of a face $F_{1}\subset{\mathcal{S}_{P}}$ etc.
  \end{definition}
  
  \begin{definition}
  A chain $\mathsf{W}_{Q}$  at $Q\in{\{\mathcal{S}_{P}\}}$ is $\epsilon$-admissible when all compositions of mappings defined by the faces along the edges of its loop $\mathsf{L}_{Q}$ are well defined.
  \end{definition}
  
  \begin{lemma}\hfill
  Let $P$ be a head of an exceptional spin graph $\mathcal{S}_{P}$ of order $r=2$.  Let $\mathcal{A}_{P}={\widetilde{P}^{i}P^{k_{1}}_{1}P^{k_{2}}_{2}}$ with $1\leq{k_{0}\leq{k_{1}}\leq{k_{2}}}$; $i+k_{1}+k_{2}=g\geq{3}$.
  \begin{enumerate}
  \item When $k_{l}=1$ for $l=0,1$ (and hence $k_{2}\geq{2}$) then the groups $\mathtt{G}_{P_{2}}$ and $\mathtt{G}_{\widetilde{P_{2}}}$ are both  isomorphic to the cyclic group $\mathtt{C}_{3}$ and all spin groups attached to the remaining vertices of $\mathcal{S}_{P}$ are trivial.
  \item When $k_{0}=1$ but $k_{l}\geq{2}$ for $l=1,2$, then each group $\mathtt{G}_{P_{l}}\cong{\mathtt{G}_{\widetilde{P_{l}}}}$ is isomorphic to the symmetric group $\mathtt{S}_{3}$ and the groups $\mathtt{G}_{P}\cong{\mathtt{G}_{\widetilde{P}}}$ are trivial.
  \item When $k_{l}\geq{2}$ for $l=0,1,2$ then the spin group $\mathtt{G}_{Q}$ attached to any vertex $Q$ of $\mathcal{S}_{P}$ is isomorphic to the symmetric group $\mathtt{S}_{3}$.
  \end{enumerate}
  \end{lemma}
  \begin{proof}
  We notice that when $F$ is a quadrilateral face of the connection graph for $\mathcal{S}_{P}$ and when $Q_{1}$ and $Q_{2}$ are two of its vertices with $deg_{\epsilon}Q_{1}\leq{deg_{\epsilon}Q_{2}}$ then the natural mapping $\widehat{F}:{\widehat{Q_{1}}}\rightarrow{\widehat{Q_{2}}}$ is either a  bijection (when $deg_{\epsilon}Q_{1}={deg_{\epsilon}Q_{2}}$) or it is an embedding (when $deg_{\epsilon}Q_{1}=r < {deg_{\epsilon}Q_{2}}=r+1=3$). In the latter case $\widehat{F}$ determines a bijection between the set $\widehat{Q_{1}}\subset{\Theta_{\epsilon}}$ and the set  $\widehat{\widehat{Q_{2}}}\subset{\widehat{Q_{2}}}\subset{\Theta_{\epsilon}}$, where $\widehat{\widehat{Q_{2}}}$ is defined as $\widehat{Q_{2}}-{\{\Phi_{Q_{2}}{\widetilde{Q_{2}}}\}}$.

  $\mathbf{Ad(1)}$. The connection graph is as on Pict.5a.  We see that it has  only two quadrilateral faces: $F=PP_{1}{\widetilde{P_{2}}}P_{2}$ and $F'=\widetilde{P}{\widetilde{P_{1}}}P_{2}{\widetilde{P_{2}}}$. Face $F$ produces the mappings
  \begin{equation*}
  \widehat{F}: \widehat{P}\leftrightarrow{\widehat{P_{1}}}\rightarrow{\widehat{\widetilde{P_{2}}}}\leftrightarrow{\widehat{P_{2}}}
  \end{equation*}
  which  are bijections (denoted by $\leftrightarrow$) when  the appropriate sets are of the same cardinality.  The mapping denoted by $\rightarrow$ is merely an embedding.  More precisely, $\widehat{P}_{1}\rightarrow{\widehat{\widetilde{P_{2}}}}$ consists of the bijection
  \begin{equation*}
    {\widehat{P_{1}}}\leftrightarrow{\{\Phi_{\widetilde{P_{2}}}P_{1},\Phi_{\widetilde{P_{2}}}P_{2}\}}\subset{\widehat{\widetilde{P_{2}}}}
  \end{equation*}
     followed by the embedding of the latter set into $\widehat{\widetilde{P_{2}}}$. From now on we will 
      write $\frac{Q}{R}$ instead of $\Phi_{R}Q$.
     
    Using this notation we may say that $\widehat{F}$ identifies
  \begin{equation} \frac{P_{1}}{P}\Leftrightarrow{\frac{P}{P_{2}}}\Leftrightarrow{\frac{P_{2}}{\widetilde{P_{2}}}}\Leftrightarrow{\frac{\widetilde{P_{2}}}{P_{1}}}; \quad  
   \frac{P_{2}}{P}\Leftrightarrow{\frac{P}{P_{1}}}\Leftrightarrow{\frac{P_{1}}{\widetilde{P_{2}}}}\Leftrightarrow{\frac{\widetilde{P_{2}}}{P_{2}}};\quad \frac{\widetilde{P}}{P_{2}}\Leftrightarrow{\frac{P_{1}}{\widetilde{P_{2}}}}
   \end {equation} 
   Similarly, face $F'$ determines the mappings
   \begin{equation*}
   \widehat{F'}: {\widehat{\widetilde{P_{1}}}}\leftrightarrow{\widehat{\widetilde{P}}}\rightarrow{\widehat{\widetilde{P_{2}}}}\leftrightarrow{\widehat{P_{2}}}
   \end{equation*}
   and   we obtain identifications analogous to ones given by $(3.2)$. Since the vertices of the spin graph $\mathcal{S}_{P}$ have different $\epsilon$-degrees not all chains in the graph are $\epsilon$-admissible.  We observe that the composition $\widehat{F'}\circ{\widehat{F}}$ acting on $\widehat{P_{2}}\cong{\{1,2,3\}_{P_{2}}}$ along the  loop $\mathsf{L}_{P_{2}}={P_{2}\widetilde{P_{2}}P_{2}}$ produces the permutation $(123)$, whereas the composition $\widehat{F}\circ{\widehat{F'}}$ along the same loop, produces the inverse permutation $(132)$. Thus we obtain that $\mathtt{G}_{P_{2}}\cong{\mathtt{C}_{3}}={\mathtt{A}_{3}}$. The similar situation occurs when we consider  the vertex $\widetilde{P_{2}}$.  In other words we will have $\mathtt{G}_{\widetilde{P_{2}}}\cong{\mathtt{C}_{3}}$ as well. On the other hand, any admissible spin chain at the remaining vertices has to contain edges only of a sigle face. Consequently, the spin groups attached to these vertices have to be trivial.

  $\mathbf{Ad(2)}$. The connection graph is given by the Pict.5b. It has the following quadrilateral faces:
  \begin{equation*}
  F_{1}=PP_{1}{\widetilde{P_{1}}}P_{2},\quad F_{1}'=\widetilde{P}\widetilde{P_{1}}P_{1}{\widetilde{P_{2}}} \quad F_{2}=P_{2}{\widetilde{P_{2}}}P_{1}P, \quad F_{2}'=\widetilde{P_{2}}P_{2}{\widetilde{P_{1}}}{\widetilde{P}}
  \end{equation*}
  and $F_{3}=P_{1}{\widetilde{P_{1}}}P_{2}{\widetilde{P_{2}}}=F_{3}'$. Only the vertices of $F_{3}$ have the same $\epsilon$-degrees and hence $F_{3}$ determines bijections between all sets $\widehat{Q}$,  $Q\in{\{F_{3}\}}$. 
  
  The vertices of the remaining faces have different $\epsilon$-degrees and the mappings which they determine are analogous to those considered  in the previous case $(1)$ above.  Let as consider an $\epsilon$-admissible chain at $\widetilde{P_{1}}$ that produces the bijections
  \begin{equation*} {\widehat{\widetilde{P_{1}}}}\stackrel{\widehat{F_{1}}}{\rightarrow}{\widehat{P_{1}}}\stackrel{\widehat{F_{3}}}{\rightarrow}{\widehat{\widetilde{P_{1}}}}
  \end{equation*} 
  We may check that the composition ${\widehat{F_{3}}}\circ{\widehat{F_{1}}}$ of this mappings produces the bijection of the set ${\widehat{\widetilde{P_{1}}}}\cong{\{1,2,3\}_{\widetilde{P_{1}}}}$ which corresponds to the permutation $(12)\in{\mathtt{S}_{3}(\widehat{\widetilde{P_{1}}})}$. When we consider the same loop $\mathsf{L}_{\widetilde{P_{1}}}$ but, instead of the face $F_{1}$ along the edge $\widetilde{P_{1}}P_{1}$, we will choose the face $F_{1}'$ then the composition ${\widehat{F_{3}}}\circ{\widehat{F_{1}'}}$ produces the permutation $(13)\in{\mathtt{S}_{3}(\widehat{\widetilde{P_{1}}})}$. This implies that the spin group $\mathtt{G}_{\widetilde{P_{1}}}$ is isomorphic to the symmetry group $\mathtt{S}_{3}$. We proceed similarly with the remaining vertices whose $\epsilon$-degree is equal to $r+1=3$ and we obtain that $\mathtt{G}_{P_{l}}\cong{\mathtt{G}_{\widetilde{P}_{l}}}\cong{\mathtt{S}_{3}}$ for $l=1,2$.  On the other hand, since the $\epsilon$-admissible loops at $P$ as well as at ${\widetilde{P}}$ are trivial, we have $\mathtt{G}_{P}\cong{\mathtt{G}_{\widetilde{P}}}\cong{\mathbb{I}}$.

  $\mathbf{Ad(3)}$.  Now $k_{l}\geq{2}$ for all $l=0,1,2$. The connection graph  ( given by the  Pict.5c)   is symmetric with respect to all of its vertices.  Hence,since   any  vertex $Q$ has the $\epsilon$-degree  equal to $r+1=3$,  all lopps $\mathsf{L}_{Q}$ in $\mathcal{S}_{P}$  and all chains are  $\epsilon$-admissible.  Let us consider (for example) the face  $F_{1}=P{\widetilde{P}}{\widetilde{P_{2}}}P_{1}$ which determines bijections
  \begin{equation*}
  \widehat{F}_{1}: \widehat{P}\leftrightarrow{\widehat{\widetilde{P}}}\leftrightarrow{\widehat{\widetilde{P_{2}}}}\leftrightarrow{\widehat{P_{1}}}
  \end{equation*}
  and face $F_{2}=P{\widetilde{P}}{\widetilde{P_{2}}}P_{2}$  producing the bijections
  \begin{equation*}
  \widehat{F_{2}}: {\widehat{P}}\leftrightarrow{\widehat{\widetilde{P}}}\leftrightarrow{\widehat{\widetilde{P_{2}}}}\leftrightarrow{\widehat{P_{2}}}
  \end{equation*}
   We see that the spin chain
   \begin{equation*}
    {\widehat{P_{1}}}\stackrel{\widehat{F_{1}}}{\rightarrow}{\widehat{\widetilde{P}}}\stackrel{\widehat{F_{2}}}{\rightarrow}{\widehat{P_{1}}}
    \end{equation*}
     produces the odd permutation $(23)$  of the set $\widehat{P_{1}}\cong{\{1,2,3\}_{P_{1}}}$, i.e.$(23)\in{\mathtt{S}_{3}(\widehat{P_{1}})}$  whereas, for example,  the spin-chain 
     \begin{equation*}
      {\widehat{P}}\stackrel{\widehat{F_{1}}}{\rightarrow}{\widehat{\widetilde{P_{2}}}}\stackrel{\widehat{F_{2}}}{\rightarrow}{\widehat{P}}
    \end{equation*}
       produces the even permutation  $(123)$ of the set ${\widehat{P}}\cong{\{1,2,3\}_{P}}$ which belongs to ${\mathtt{S}_{3}(\widehat{P})}$. Thus, for any vertex $Q$ of the spin graph $\mathcal{S}_{P}$,  by considering a few spin-chains at this vertex $Q$,  we obtain generators of the whole symmetry group acting on the set ${\widehat{Q}}\cong{\{1,2,3\}_{Q}}$. Summarizing, when the connection graph is as on Pict.5c then  the spin group   $\mathtt{G}_{Q}$ associated to each vertex $Q\in{\{\mathcal{S}_{P}\}}$ is $\mathtt{G}_{Q}\cong{\mathtt{S}_{3}}$.
   
  \end{proof}

\section{SPIN GROUPS FOR GRAPHS OF DEGREE $r=3$ }
When the degree $r$ of an exceptional spin graph $\mathcal{S}_{P}$ is greater than two then the situation changes drastically. Namely, when $r=2$ then any additional edge between conjugate vertices  "`kills"' the unique cell of the basic graph $\mathcal{S}(r)={\mathcal{S}(2)}\subset{\mathcal{S}_{P}}$. In the contrary to this, when $r\geq{3}$ then each additional edge between conjugate vertices produces faces that are additional  to the faces of the basic graph. 

Let $P$ be a head of an exceptional spin  graph $\mathcal{S}_{P}$ of order $r=3<g$.  This means that the integral divisor $\mathcal{A}_{P}$ (determined by the section $\sigma_{P}$ of $\xi_{\epsilon}$)  has the form $\mathcal{A}_{P}={\widetilde{P}^{i}}P^{k_{1}}_{1}P^{k_{2}}_{2}{P^{k_{3}}_{3}}$.  We must have  $1\leq{k_{0}=i+1}\leq{k_{1}}\leq{k_{2}}\leq{k_{3}}$  and $i+k_{1}+k_{2}+k_{3}=g\geq{4}$.  As before, we will use tha notation: $k_{1}=k_{0}+p_{1}$, $k_{2}=k_{1}+p_{2}$, and $k_{3}=k_{2}+p_{3}$ which will allow us to write the class of graphs isomorphic to $\mathcal{S}_{P}$ as $\mathcal{S}_{i,p_{1},p_{2},p_{3}}$. Moreover we have
   \begin{equation}
   0\leq{i}\leq{\left\lfloor \frac{g-3}{4}\right\rfloor} \quad \text{and}\quad g=4i+3p_{1}+2p_{2}+p_{3}+3
   \end{equation}
   Since the degree of an exceptional spin graph $\mathcal{S}_{P}$ is $r=3$ its basic graph has the same form as the standard spin graph on a surface of genus $3$. Let $F$ be any face of the basic graph  i.e. let $F\subset{\mathcal{S}(3)}\subset{\mathcal{S}_{P}}$.  Suppose that $F=Q_{1}Q_{2}Q_{3}Q_{4}$.    Let ${\widehat{\widehat{Q_{l}}}}={\widehat{Q_{l}}}$ when $deg_{\epsilon}Q_{l}=r=3$ and let ${\widehat{\widehat{Q_{l}}}}={\widehat{Q_{l}}-{\frac{\widetilde{Q_{l}}}{Q_{l}}}}$  when   $Q_{l}$ has  $\epsilon$-degree  equal to $r+1=4$.
       The mappings ${\widehat{F}}$ restricted to the sets $\widehat{\widehat{Q_{l}}}$, $l=1,2,3,4$, are exactly the same as in a standard case.  More precisely,  we have
 \begin{equation}
 \widehat{F}:  {\widehat{\widehat{Q_{1}}}}\leftrightarrow{\widehat{\widehat{Q_{2}}}}\leftrightarrow{\widehat{\widehat{Q_{3}}}}\leftrightarrow{\widehat{\widehat{Q_{4}}}}
 \end{equation}

    Moreover, for any  two  vertices of $F$, say $Q_{l}$ and $Q_{k}$, $l\neq{k}$, with $deg_{\epsilon}Q_{l}=deg_{\epsilon}Q_{k}=4$   the bijection ${\widehat{\widehat{Q_{l}}}}\stackrel{\widehat{F}}{\leftrightarrow}{\widehat{\widehat{Q_{k}}}}$ can be naturally extended to   the (also denoted by $\widehat{F}$) bijection
   \begin{equation}
   \widehat{F}: {\widehat{Q_{l}}}\leftrightarrow{\widehat{Q_{k}}} \quad \text{with} \quad {\frac{\widetilde{Q_{l}}}{Q_{l}}}\Leftrightarrow{\frac{\widetilde{Q_{k}}}{Q_{k}}}
   \end{equation}
   In other words, any face $F\subset{\mathcal{S}(3)}\subset{\mathcal{S}_{P}}$ produces the bijection ${\widehat{Q_{i}}}\leftrightarrow{\widehat{Q_{j}}}$
   for any two of its vertices $Q_{i}$ and $Q_{j}$ with the same $\epsilon$-degree.  Besides, we see that all bijections and all groups of permutations produced by the basic graph $\mathcal{S}(r)\subset{\mathcal{S_{P}}}$ are inherited by the sets $\widehat{Q}$ (when $deg_{\epsilon}Q=r$) or by the sets $\widehat{\widehat{Q}}={\widehat{Q}}-{\frac{\widetilde{Q}}{Q}}$ (when $deg_{\epsilon}Q={r+1}$) respectively.
\begin{definition}
 Let ${\mathsf{W}_{Q}}=({\mathsf{L}_{Q}};F_{1},\ldots,F_{N+1})$ be a spin chain at a vertex $Q$.  When  $\mathsf{L}_{Q}\subset{\mathcal{S}(r)}$ as well as when all faces $F_{l}\subset{\mathcal{S}(r)}$ for  $l=1,\ldots,N+1$, then we say that the spin chain $\mathsf{W}_{Q}$ is a basic spin chain.
 \end{definition}
 Using $(4.2)$ and exactly the same considerations as in [2] we see that all basic chains at any vertex $Q$ of $\mathcal{S}_{P}$ produce the subgroup of the spin group $\mathtt{G}_{Q}$ which is isomorphic to the alternating group $\mathtt{A}_{3}$ (i.e. to the standard spin group of genus $3$) acting on the set $\widehat{\widehat{Q}}\subseteq{\widehat{Q}}$.

When a spin graph is a standard or a Weierstrass graph then each spin- chain $\mathsf{W}_{Q}=({\mathsf{L}_{Q}}; F_{1},{\ldots},F_{N+1})$ is $\epsilon$-admissible (so we do not even mention the word "`admissibility"'). The situation is different for exceptional graphs. When a graph has degree $r=2$ then a chain at a vertex $Q$ is $\epsilon$-admissible  when all of its vertices have the same $\epsilon$-degree.

For exceptional graphs of degree $r\geq{3}$ the situation is more complicated.
 Suppose that face $F$ does not belong to the basic graph of $\mathcal{S}_{P}$. This means that at least one or possibly two pairs of vertices of $F$ are mutually conjugate (we notice that the degree of each of these conjugate vertices must be $4$).  When set $\{F\}$ of vertices contains only one pair of conjugate points, say $\{S,\widetilde{S}\}\subset{\{F\}}$, then for any vertex $Q$ of $F$ we will define the set ${{\widehat{Q}}^{F}}\subseteq{\widehat{Q}}$ as follows:
\begin{definition}
Let F be a face of an exceptional spin graph of degree $r=3$ with vertices ${\{F\}}={\{R_{1},R_{2},R_{3},{\widetilde{R_{3}}}\}}$,  where $R_{1}$ and $R_{2}$ are not conjugate. Let $F'$ denote the face whose vertices are ${\{F'\}}={\{{\widetilde{R_{1}}},{\widetilde{R_{2}}},{{R_{3}}},{\widetilde{R_{3}}}\}}$. For each $l=1,2$, when $deg_{\epsilon}R_{l}=3$ then 
\begin{equation*}
{\widehat{R_{l}}}^{F}={\widehat{R_{l}}} \quad \text{and} \quad {\widehat{\widetilde{R_{l}}}}^{F'}={\widehat{\widetilde{R_{l}}}}
\end{equation*}
 For each $l=1,2,3$,  when  $deg_{\epsilon}R_{l}=4$, then 
\begin{equation*}
{\widehat{R_{l}}}^{F}={\widehat{R_{l}}}-{\frac{\widetilde{R_{4}}}{R_{l}}} \quad \text{and} \quad {\widehat{\widetilde{R_{l}}}}^{F}={\widehat{\widetilde{R_{l}}}}-{\frac{R_{4}}{\widetilde{R_{l}}}}
\end{equation*}
Here, the set $\{R_{4},{\widetilde{R_{4}}}\}$ of verices is uniquely determined by the condition that its intersection  with $\{F\}$ as well as  with $\{F'\}$ is empty (i.e. ${\{\mathcal{S}_{P}\}}={\{R_{4},{\widetilde{R_{4}}}\}}\cup{\{F\}}\cup{\{F'\}}$). 

Besides,  whenever $R_{l}\in{F\cap{F'}}$ then ${\widehat{R_{l}}}^{F}={\widehat{R_{l}}}^{F'}$ and the same is true for $\widetilde{R_{l}}$.

\end{definition}
When  face $F$ has two pairs of mutually conjugate points as the set of its vertices then  all of its vertices must have $\epsilon$-degree equal to $4$. 
\begin{definition}
Suppose that the set  of vertices of a face $F\subset{\mathcal{S}_{P}}$ consists of two pairs of conjugate points, say, ${\{F\}}={\{S_{1},{\widetilde{S_{1}}},S_{2},{\widetilde{S_{2}}}\}}$. Let ${\{S_{3},S_{4},{\widetilde{S_{3}}},
{\widetilde{S_{4}}}\}}$ denote the set of the remaining vertices of the spin graph $\mathcal{S}_{P}$.  We will define bijections
\begin{equation}
 {\widehat{F}}: {\widehat{S_{1}}}\leftrightarrow{\widehat{\widetilde{S_{1}}}}\leftrightarrow{\widehat{S_{2}}}\leftrightarrow{\widehat{\widetilde{S_{2}}}}
 \end{equation}
 by assuming the following, quite obvious correspondences:
\begin{equation*}
{\widehat{F}}:  {\frac{\widetilde{S_{3}}}{S_{1}}}\Leftrightarrow{\frac{S_{3}}{\widetilde{S_{2}}}}\Leftrightarrow{\frac{S_{3}}{\widetilde{S_{1}}}}\Leftrightarrow{\frac{\widetilde{S_{3}}}{S_{2}}}\quad \text{and} \quad {\frac{S_{4}}{\widetilde{S_{1}}}}\Leftrightarrow{\frac{\widetilde{S_{4}}}{S_{2}}}\Leftrightarrow{\frac{S_{4}}{\widetilde{S_{2}}}}\Leftrightarrow{\frac{\widetilde{S_{4}}}{S_{1}}}
\end{equation*}
 together with the natural mapings 
\begin{equation*}
{\widehat{F}}: {\frac{S_{1}}{\widetilde{S_{1}}}}\Leftrightarrow{\frac{\widetilde{S_{1}}}{S_{2}}}\Leftrightarrow{\frac{S_{2}}{\widetilde{S_{2}}}}\Leftrightarrow{\frac{\widetilde{S_{2}}}{S_{1}}} \quad \text{and} \quad  {\frac{S_{1}}{\widetilde{S_{2}}}}\Leftrightarrow{\frac{\widetilde{S_{2}}}{S_{2}}}\Leftrightarrow{\frac{S_{2}}{\widetilde{S_{1}}}}\Leftrightarrow{\frac{\widetilde{S_{1}}}{S_{1}}}
\end{equation*}
which we read from the connection graph for $\mathcal{S}_{P}$.
 \end{definition}

Let $Q$ be any vertex of an exceptional spin graph $\mathcal{S}_{P}$ of degree $r=3<g$. Let $\mathsf{W}_{Q}$ be a chain at $Q$ of the form
\begin{equation}
{\mathsf{W}_{Q}}:\quad {Q}\stackrel{F_{1}}{\rightarrow}{Q_{1}}\stackrel{F_{2}}{\rightarrow}{Q_{2}}\stackrel{F_{3}}{\rightarrow}{\ldots}\stackrel{F_{N}}{\rightarrow}{Q_{N}}\stackrel{F_{N+1}}{\rightarrow}{Q}  
\end{equation}
where faces above the arrows indicate choices along appropriate edges respectively. We should consider the following two cases separately:

$\mathbf{CASE}$ $\mathbf{I}$. All vertices of an exceptional graph $\mathcal{S}_{P}$ have degree 4 i.e. the connection graph for $\mathcal{S}_{P}$ is totally symmetric with all mutually conjugate vertices connected. This implies that for any face $F\subset{\mathcal{S}_{P}}$ and for any of its edge $RS\subset{F}$ we have well defined bijection ${\widehat{F}}: {\widehat{R}}\leftrightarrow{\widehat{S}}$. Hence, for all chains $(4.5)$ we have ${\widehat{Q_{l}}}={\{1,2,3,4\}_{Q_{l}}}$ for $l=1,\ldots,N$, and ${\widehat{Q}}={\{1,2,3,4\}_{Q}}$ and all compositions of indicated isomorphisms between the appropriate $4$-element sets are well defined.
\begin{equation}
{\widehat{\mathsf{W}_{Q}}}: {\widehat{Q}}\stackrel{\widehat{F_{1}}}{\rightarrow}{\widehat{Q_{1}}}\stackrel{\widehat{F_{2}}}{\rightarrow}{\widehat{Q_{2}}}\stackrel{\widehat{F_{3}}}{\rightarrow}{\ldots}\stackrel{\widehat{F_{N}}}{\rightarrow}{\widehat{Q_{N}}}\stackrel{\widehat{F_{N+1}}}{\rightarrow}{\widehat{Q}}
\end{equation}
Hence, all chains are admissible and each of them produces a concrete permutation $\sigma^{\mathsf{W}_{Q}}\in{\mathtt{S}_{4}(\widehat{Q})}$.

$\mathbf{CASE}$ $\mathbf{II}$. Not all vertices of $\mathcal{S}_{P}$ have degree equal to $4$, i.e. there is at least one pair of conjugate vertices that is not connected in $\mathcal{S}_{P}$. Now we can consider   isomorphisms determined by faces $F_{l}$, $l=1,\ldots,N$, in the sequence $(4.5)$  only as the isomorphisms between appropriate   $3$ element sets  ${\widehat{Q_{l}}^{F_{l}}}$, see remark below.  More precisely, the spin groups attached to the vertices of the graph $\mathcal{S}_{P}$ are produced only by $\epsilon$-admissible chains.  When chain $(4.5)$ is not admissible then we will assume that it does not move the elements of $\widehat{Q}$.
\newtheorem{remark}{Remark}
\begin{remark}\hfill
Let $Q$ be a vertex of an exceptional spin graph $\mathcal{S}_{P}$ of degree $r={3}$.  Let $\mathsf{W}_{Q}$ be a spin-chain $\mathsf{W}_{Q}=({\mathsf{L}_{Q}}; F_{1},{\ldots},F_{N+1})$ at $Q$  as in $(4.5)$.  
\begin{itemize}
\item When all vertices of $\mathcal{S}_{P}$ have degree equal to $4$ then any chain $(4.5)$ is admissible. 
\end {itemize}
Suppose that for each face $F_{l}$ in $(4.5)$  the set $\{F_{l}\}$ of its vertices consist of at most one pair of mutually conjugate points. 
\begin{itemize}
\item  If the degree of a vertex $Q$ is equal to $3$ then a chain $\mathsf{W}_{Q}$ at a vertex $Q$  is admissible when  we have ${\widehat{Q_{l}}}^{F_{l}}={\widehat{Q_{l}}}^{F_{l+1}}$  for any $l=1,2,\ldots,N$.
\item   When ${deg_{\epsilon}Q}=4$ then a chain $(4.5)$ at $Q$ is admissible if ${\widehat{Q_{l}}}^{F_{l}}={\widehat{Q_{l}}}^{F_{l+1}}$ for any $l=1,2,\ldots,N-1$.
\end{itemize}
\end{remark}

We observe that in the latter case   we may have  ${\widehat{Q_{N}}}^{F_{N}}\neq{\widehat{Q_{N}}}^{F_{N+1}}$.  Since this is possible only when the degree of $Q_{N}$ is the same as the degree of $Q$ (i.e. it is equal to $4$) our chain $\mathsf{W}_{Q}$ produces an isomorphism (given by appropriate compositions) between the $3$-elements  sets ${\widehat{Q}^{F_{1}}}$ and ${\widehat{F_{N}}(\widehat{Q_{N}})^{F_{N}}}\subset{\widehat{Q}}$   wich extends to the isomorphism ${\widehat{Q}}\rightarrow{\widehat{Q}}$  in the obvious way.

Let $\mathcal{S}_{P}$ be an exceptional spin graph of  order $r=3$ with a head $P$ and with ${\widehat{k}}={\widehat{k}(P)}={(k_{0},k_{1},k_{2},k_{3})}$. To find the spin groups attached to the vertices of $\mathcal{S}_{P}$ we will proceed by considering different possible connection graphs for $\mathcal{S}_{P}$ separately.

   \subsection{All $k_{l}$, for $l=0,1,2$ are equal to $1$.} In this case we must have $k_{3}={g-2}\geq{2}$ and the $\epsilon$-degrees of $P_{3}$ and $\widetilde{P_{3}}$ must be $4$. The degrees af all remaining vertices  of the spin graph are equal to $3$.  The connection graph for $\mathcal{S}_{P}$ has the form showed by the Pict.7. The faces 
   \begin{equation}
   F_{1}=PP_{1}{\widetilde{P_{3}}}P_{2};\quad F_{2}=PP_{1}{\widetilde{P_{2}}}P_{3};\quad F_{3}=P_{1}{\widetilde{P_{3}}}{\widetilde{P}}{\widetilde{P_{2}}}
   \end{equation}
   as well as and the faces
   \begin{equation}
   F_{1}'={\widetilde{P}}{\widetilde{P_{1}}}P_{3}{\widetilde{P_{2}}}; \quad F_{2}'={\widetilde{P}}{\widetilde{P_{1}}}P_{2}{\widetilde{P_{3}}}; \quad F_{3}'={\widetilde{P_{1}}}P_{3}PP_{2}
   \end{equation}
   are the faces of the basic graph $\mathcal{S}(3)\subset{\mathcal{S}_{P}}$ and  all mappings determined by them were already discussed above (see $(4.2)$) and $(4.3)$. However, the additional edge $P_{3}\widetilde{P_{3}}$ of $\mathcal{S}_{P}$ gives rise to the following additional faces for the connection graph of $\mathcal{S}_{P}$:
   \begin{equation}
   F_{4}=PP_{1}{\widetilde{P_{3}}}P_{3} \quad F_{5}=PP_{2}{\widetilde{P_{3}}}P_{3} \quad F_{6}=P_{1}{\widetilde{P_{2}}}P_{3}{\widetilde{P_{3}}}
   \end{equation} 
   and
   \begin{equation}
   F_{4}'={\widetilde{P}}{\widetilde{P_{1}}}P_{3}{\widetilde{P_{3}}} \quad F_{5}'={\widetilde{P}}{\widetilde{P_{2}}}P_{3}{\widetilde{P_{3}}} \quad F_{6}'={\widetilde{P_{1}}}P_{2}{\widetilde{P_{3}}}P_{3}
   \end{equation}

   \begin{pspicture}(-2,-2)(6,4)
   \psline[showpoints=true, linestyle=dashed]%
   (-1,0)(1,0)(2,1)
   \psline[showpoints=true, linestyle=dashed]%
   (-1,2)(1,2)(2,3)
   \psline[showpoints=true, linestyle=dashed]%
   (0,3)(0,1)
   \psline[linestyle=dashed]%
   (-1,0)(0,1)(2,1)
   \psline[linestyle=dashed]%
   (-1,2)(0,3)(2,3)
   \psline[linestyle=dashed]%
   (1,0)(1,2)
   \psline[linestyle=dashed]%
   (2,1)(2,3)
   \psline[linestyle=dashed]%
   (-1,0)(-1,2)
   \psline[linestyle=dashed]%
   (-1,2)(2,1)
   
   \rput(-0.1,-1){\rnode{A}{Pict.7}}
   \rput(4.5,2){\rnode{B}{$k_{0}={k_{1}}={k_{2}}=1$}}
   \rput(4.5,1.5){\rnode{C}{$k_{3}={g-3}\geq{2}$}}
   
   \rput(-0.3,3.1){\rnode{a}{$\widetilde{P_{1}}$}}
   \rput(2.3,3.1){\rnode{b}{$\widetilde{P}$}}
   \rput(-0.3,1.2){\rnode{c}{$P_{2}$}}
   \rput(2.3,1.2){\rnode{d}{$\widetilde{P_{3}}$}}
   \rput(-1.3,2.2){\rnode{e}{$P_{3}$}}
   \rput(0.7,2.3){\rnode{f}{$\widetilde{P_{2}}$}}
   \rput(-1.3,-0.1){\rnode{g}{$P$}}
   \rput(1.3,-0.1){\rnode{h}{$P_{1}$}}

   \end{pspicture}

  Let us consider an admissible (but not a basic) chain $\mathsf{W}_{P}=({\mathsf{L}_{P}}=PP_{2}P; F_{5},F_{4})$. It produces the mappings
  \begin{equation}
  {\widehat{P}}\stackrel{\widehat{F_{5}}}{\rightarrow}{\widehat{P_{2}}}\stackrel{\widehat{F_{4}}}{\rightarrow}{\widehat{P}}
  \end{equation}
 whose composition corresponds to the odd permutation $(13)\in{\mathtt{S}_{3}(\widehat{P})}$ of the set ${\widehat{P}}={\{{\frac{P_{1}}{P}},{\frac{P_{2}}{P}},{\frac{P_{3}}{P}}\}}\cong{\{1,2,3\}_{P}}$. Since basic chains produce a group of permutations of the set $\widehat{P}$ isomorphic to the alternating group we see that the whole spin group $\mathtt{G}_{P}={\mathtt{S}_{3}}(\widehat{P})$ is isomorphic to the whole symmetry group $\mathtt{S}_{3}$.
 
  For any vertex $Q$ of $\mathcal{S}_{P}$ with  the degree $deg_{\epsilon}Q=3=r$ the situation is analogous.  Hence,for each $Q\in{\{\mathcal{S}_{P}\}}$ with $Q\notin{\{P_{3},{\widetilde{P_{3}}}\}}$  the spin group $\mathtt{G}_{Q}$ is  isomorphic to the symmertry group $\mathtt{S}_{3}$ 
  
 On the other hand, it is easy to see that a non-admissible chain at $P$ with ${\mathsf{L}_{P}}={PP_{3}P}$ and with faces $F_{5}$ and $F_{2}$ along its edges respectively, does not lead to any permutation of the set ${\{1,2,3\}_{P}}\cong{\widehat{P}}$. The reason for this is that the bijection  ${\widehat{F_{5}}}:{\widehat{P}}\leftrightarrow{\widehat{P_{3}}}^{F_{5}}$ does not allow for any composition with ${\widehat{F_{2}}}:{\widehat{\widehat{P_{3}}}}\leftrightarrow{\widehat{P}}$ (in other words, we have ${\widehat{P_{3}}}^{F_{5}}\neq{\widehat{P_{3}}}^{F_{2}}$ ).
 
 Now let us find the spin group at the vertex $P_{3}$. This point has degree 4 and we have
 \begin{equation*} {\widehat{P_{3}}}={\{{\frac{P}{P_{3}}},{\frac{\widetilde{P_{1}}}{P_{3}}},{\frac{\widetilde{P_{2}}}{P_{3}}},{\frac{\widetilde{P_{3}}}{P_{3}}}\}}\cong{\{1,2,3,4\}_{P_{3}}} 
 \end{equation*}
 Let us consider an admissible chain at $P_{3}$ given by ${\mathsf{W}_{P_{3}}}=({\mathsf{L}_{P_{3}}}={P_{3}\widetilde{P_{3}}P_{3}};F_{4},F_{5})$. Since each of these faces has only one pair of conjugate vertices, we use the definition $6$ to find the appropriate bijections.  We find that this particular chain produces a permutation $(234)$ of the set $\{1,2,3,4\}_{P_{3}}\cong{\widehat{P_{3}}}$. For some other admissible chain, for example for $\mathsf{W}_{P_{3}}=({\mathsf{L}_{P_{3}}=P_{3}{\widetilde{P_{1}}}{\widetilde{P}}{\widetilde{P_{3}}}P_{3}};F_{1}',F_{1}',F_{3},F_{5})$ we obtain an odd permutation $(1234)$ of the set $\widehat{P_{3}}$. Hence, since all basic chains produce the alternating group acting on the set ${\widehat{\widehat{P_{3}}}}\subset{\widehat{P_{3}}}$, we see that the spin group at $P_{3}$ is the whole symmetry group, i.e. $\mathtt{G}_{P_{3}}={\mathtt{S}_{4}}(\widehat{P_{3}})$. 
 
 We will find the similar result for the spin group at $\widetilde{P_{3}}$. Summarizing, when $k_{l}=1$ for all $l=0,1,2$ then the spin groups attached to any vertex $Q$ of an exceptional spin graph $\mathcal{S}_{P}$ of degree $r=3<g$ are:
 \begin{equation}
 \begin{cases}
 \mathtt{S}_{3}, &\text{when} \quad Q\notin{\{P_{3},\widetilde{P_{3}}\}}\\
 \mathtt{S}_{4}, &\text{when} \quad Q{\{P_{3},\widetilde{P_{3}}\}}
 \end{cases}
 \end{equation}
 
 \subsection{Assume that $k_{0}=k_{1}=1<k_{2}\leq{k_{3}}$}. In this case two pairs of conjugate vertices are connected and the connection graph of $\mathcal{S}_{P}$ is given by the Pict.8.  Now, besides of the faces $F_{l}$ and $F_{l}'$ for $l=1,\ldots,6$ given by $(4.7)-(4.10)$ the connection graph for $\mathcal{S}_{P}$ has the following, additional faces:
 \begin{equation}
 F_{7}=PP_{3}{\widetilde{P_{2}}}P_{2},\quad F_{8}=PP_{1}{\widetilde{P_{2}}}P_{2}, \quad F_{9}=P_{1}{\widetilde{P_{3}}}P_{2}{\widetilde{P_{2}}} \quad F_{10}=P_{3}{\widetilde{P_{2}}}P_{2}{\widetilde{P_{3}}}
 \end{equation}
 and 
 \begin{equation}
 F_{7}'={\widetilde{P}}{\widetilde{P_{3}}}P_{2}{\widetilde{P_{2}}} \quad F_{8}'={\widetilde{P}}{\widetilde{P_{1}}}P_{2}{\widetilde{P_{2}}} \quad F_{9}'={\widetilde{P_{1}}}P_{3}{\widetilde{P_{2}}}P_{2} \quad F_{10}'=F_{10}
 \end{equation}
 
 Whenever  spin graph $\mathcal{S}_{P}$ has (at least) two pairs of  conjugate vertices  connected  then there exists  
    a face $F\subset{\mathcal{S}_{P}}$  whose set of vertices consists of two pairs of mutually conjugate points, say $\{F\}={\{S_{1},{\widetilde{S_{1}}},S_{2},{\widetilde{S_{2}}}\}}$. 
    
 \begin{pspicture}(-2,-2)(5,4)
 \psline[showpoints=true, linestyle=dashed]%
 (-1,0)(1,0)(2,1)
 \psline[showpoints=true, linestyle=dashed]%
 (-1,2)(1,2)(2,3)
 \psline[showpoints=true, linestyle=dashed]%
 (0,3)(0,1)
 \psline[linestyle=dashed]%
 (-1,2)(0,3)(2,3)
 \psline[linestyle=dashed]%
 (-1,0)(0,1)(2,1)
 \psline[linestyle=dashed]%
 (2,1)(2,3)
 \psline[linestyle=dashed]%
 (1,0)(1,2)
 \psline[linestyle=dashed]%
 (-1,0)(-1,2)
 \psline[linestyle=dashed]%
 (0,1)(1,2)
 \psline[linestyle=dashed]%
 (-1,2)(2,1)

 \rput(0.5,-1){\rnode{A}{Pict.8}}
 \rput(4.5,2){\rnode{B}{$k_{0}=k_{1}=1$}}
 \rput(4.5,1.5){\rnode{C}{$1<{k_{2}}\leq{k_{3}}$}}
 
 \rput(-0.3,3.1){\rnode{a}{$\widetilde{P_{1}}$}}
 \rput(2.3,3.1){\rnode{b}{$\widetilde{P}$}}
 \rput(-1.3,2.2){\rnode{c}{$P_{3}$}}
 \rput(0.7,2.3){\rnode{d}{$\widetilde{P_{2}}$}}
 \rput(-1.3,-0.1){\rnode{e}{$P$}}
 \rput(1.3,-0.1){\rnode{f}{$P_{1}$}}
 \rput(2.3,1.2){\rnode{g}{$\widetilde{P_{3}}$}}
 \rput(-0.3,1.2){\rnode{h}{$P_{2}$}}
 
 \end{pspicture}

    Such face $F$  does not define  the $3$-element set ${\widehat{Q}}^{F}$, for any of its vertex $Q\in{\{F\}}$. 
      However, $\widehat{F}$ determines bijections  $(4.4)$ between the all appropriate $4$-element sets $\widehat{Q}$, $Q\in{\{F\}}$. Hence,
       for any face ${\overline{F}}\subset{\mathcal{S}_{P}}$ with an edge $RQ$, ($Q\in{\{F\}}$) and with at most one pair of conjugate vertices   we have well defined composition  ${\widehat{F}}\circ{\widehat{\overline{F}}}$  . The isomorphism ${\widehat{\overline{F}}}: {\widehat{R}^{\overline{F}}}\rightarrow{\widehat{Q}}^{\overline{F}}$ has well defined composition with ${\widehat{F}}$ wich maps  the set ${\widehat{R}^{\overline{F}}}$ into an appropriate subset ${\widehat{F}}\circ{\widehat{\overline{F}}}({\widehat{R}}^{\overline{F}})$ of $\widehat{Q}$.

   In our case,  (see Pict.8)   only the face $F_{10}$ has two pairs of conjugate points as its vertices. Let
    \begin{equation}
    {\widehat{F_{10}}}: {\widehat{P_{2}}}\leftrightarrow{\widehat{\widetilde{P_{3}}}}\leftrightarrow{\widehat{P_{3}}}\leftrightarrow{\widehat{\widetilde{P_{2}}}}
    \end{equation}
  be bijections determined by $F_{10}$ according to our definition $7$. Suppose that we would like to consider a chain at $P$ given by
  \begin{equation}
  {\mathsf{W}_{P}}: P\stackrel{F_{3}'}{\rightarrow}P_{2}\stackrel{F_{10}}{\rightarrow}{\widetilde{P_{3}}}\stackrel{F_{6}}{\rightarrow}P_{1}\stackrel{F_{2}}{\rightarrow}P
  \end{equation}
  Now ${\widehat{F_{3}'}(\widehat{P})}={\widehat{\widehat{P_{2}}}}$ and ${\widehat{F_{10}}}$ determines bijections between the following $3$-element subsets :
  \begin{equation*}
  {\widehat{F_{10}}}: {\widehat{\widehat{P_{2}}}}\leftrightarrow{\widehat{\widetilde{P_{3}}}}^{F_{4}}={\widehat{\widetilde{P_{3}}}}^{F_{4}'}\leftrightarrow{\widehat{\widetilde{P_{2}}}}^{F_{8}}={\widehat{\widetilde{P_{2}}}}^{F_{8}'}\leftrightarrow{\widehat{\widehat{P_{3}}}}
  \end{equation*}
  Since we have ${\widehat{F_{10}}({\widehat{\widehat{P_{2}}}})}\neq{\widehat{\widetilde{P_{3}}}}^{F_{6}}$ the chain $(4.16)$ can not be admissible. However,  if instead of the face $F_{6}$ in $(4.16)$ we  take the face $F_{4}$ or ${F_{4}'}$ then our chain would become admissible. For example, when we  consider $F_{4}$ instead of $F_{6}$ in $(4.16)$ then the composition of isomorphisms determined by such chain is an isomorphism of the set ${\widehat{P}}\cong{\{1,2,3,\}_{P}}$ corresponding to the permutation $\sigma^{\mathsf{W}_{P}}=(123)$.

 \begin{remark}\hfill
Suppose that  spin graph $\mathcal{S}_{P}$  of degree $r=3$ has  two pairs  of conjugate vertices  that are connected. 
\begin{enumerate}
\item   Let all vertices of a loop $\mathsf{L}_{Q}$ have degrees equal to $4$. In such case, for any choices of faces along the edges a chain $\mathsf{W}_{Q}$ is admissible.
\item  The degree of $Q$ in $(4.5)$ is equal to $3$  and some $F_{l}\neq{F_{l+1}}$ has two pairs of conjugate vertices (it is possible only for $l>1$). Now, for $\mathsf{W}_{Q}$ o be admissible we must have 
\begin{equation*}
{\widehat{F_{l}}({\widehat{Q_{l-1}}}^{F_{l-1}})}={\widehat{Q_{l}}}^{F_{l+1}} \quad \text{for} \quad l=2,\ldots,N.
\end{equation*}
 When $F_{l}={F_{l+1}}$ then the latter condition becomes $\widehat{F}_{l}({\widehat{Q_{l-1}}}^{F_{l-1}})\stackrel{\widehat{F}_{l}}{\leftrightarrow}{\widehat{Q}_{l+1}}^{F_{l+2}}$.
\item  The degree of $Q$ is equal to $4$ and some of vertices of $\mathsf{L}_{Q}$ have $\epsilon$-degree equal to $3$.  Then $\mathsf{W}_{Q}$ is admissible if the same condition as in $(2)$ above are satisfied for $l=1,2,\ldots,N$.
\item If each face of $\mathsf{W}_{Q}$ has at most one pairs of conjugate vertices then we use the remark $1$.
\end{enumerate}
\end{remark}
Let us consider a chain at $P$ given by
\begin{equation*}
{\mathsf{W}_{P}}: P\stackrel{F_{3}'}{\rightarrow}P_{2}\stackrel{F_{10}}{\rightarrow}{\widetilde{P_{3}}}\stackrel{F_{10}}{\rightarrow}{P_{3}}\stackrel{F_{2}}{\rightarrow}{P}
\end{equation*}
Since ${\widehat{F_{10}}({\widehat{P_{2}}}^{F_{3}'})}\leftrightarrow{\widehat{P_{3}}}^{F_{2}}={\widehat{\widehat{P_{3}}}}$ this chain is admissible and produces an odd permutation ${\sigma^{\mathsf{W}_{P}}}=(13)$ acting on the set ${\widehat{P}}\cong{\{1,2,3\}_{P}}$.

Another admissible chain at $P$, ${\mathsf{W}_{P}}=({\mathsf{L}_{P}}={PP_{1}P};F_{8},F_{1})$ determines also the odd permutation $(13)$ acting on $\{1,2,3\}_{P}$. Since all basic chains at $P$ produce permutations of $\widehat{P}$ that form the group isomorphic to the alternating group $\mathtt{A}_{3}$ we see that the set of admissible chains at $P$ produce the whole symmetry group $\mathtt{S}_{3}(\widehat{P})$. In analogous way we may show that for any vertex $Q$ of $\mathcal{S}_{P}$ with ${deg_{\epsilon}Q}=3$ the spin group $\mathtt{G}_{Q}$ is isomorphic to the whole symmetry group, i.e. $\mathtt{G}_{Q}={\mathtt{S}_{3}}(\widehat{Q})$.

Now, let us suppose that the degree of the vertex $Q$ of a chain $\mathsf{W}_{Q}$ in $(4.5)$ is equal to $4$. For example, let $Q=P_{2}$. Let us consider
\begin{equation*}
\mathsf{W}_{P_{2}}:{P_{2}}\stackrel{F_{7}}{\rightarrow}{P}\stackrel{F_{2}}{\rightarrow}{P_{3}}\stackrel{F_{10}}{\rightarrow}{\widetilde{P_{2}}}\stackrel{F_{8}}{\rightarrow}{P_{1}}\stackrel{F_{9}}{\rightarrow}{\widetilde{P_{3}}}\stackrel{F_{10}}{\rightarrow}{P_{2}}
\end{equation*}
Since we have ${\widehat{F_{10}}({\widehat{\widehat{P_{3}}}})}={\widehat{\widetilde{P_{2}}}}^{F_{8}}$ and the compositions of all appropriate  mappings are well defined.  Our chain is admissible and produces the isomorphism of the set ${\widehat{P_{2}}}\cong{\{1,2,3,4\}_{P_{2}}}$ corresponding to the permutation ${\sigma^{\mathsf{W}_{P_{2}}}}=(243)$. For another admissible chain
\begin{equation*}
{\mathsf{W}_{P_{2}}}: {P_{2}}\stackrel{F_{10}}{\rightarrow}{\widetilde{P_{3}}}\stackrel{F_{2}'}{\rightarrow}{P_{2}}
\end{equation*}
we obtain the isomorphism of $\widehat{P_{2}}$ corresponding to the odd permutation $\sigma^{\mathsf{W}_{P_{2}}}=(1432)$ acting on the set ${\{1,2,3,4\}_{P_{2}}}\cong{\widehat{P_{2}}}$. Thus, since the basic chains produce ${\mathtt{A}_{3}(\widehat{\widehat{P_{2}}})}$ we see that the set of all admissible chains at $P_{2}$ result in the spin group ${\mathtt{G}_{P_{2}}}={\mathtt{S}_{4}}(\widehat{P_{2}})$.  For any other vertex of $\mathcal{S}_{P}$ of degree equal to $4$ we will obtain the similar result. Thus we have:
\begin{equation}
{\mathtt{G}_{Q}}\cong
\begin{cases}
\mathtt{S}_{3}, \quad \text{when}\quad {deg_{\epsilon}Q}=3\\
\mathtt{S}_{4}, \quad \text{when}\quad {deg_{\epsilon}Q}=4
\end{cases}
\end{equation}

\subsection{Let $k_{0}=1$ and let ${k_{0}}<{k_{1}}\leq{k_{2}}\leq{k_{3}}$}.
The connection graph is given by the Pict.9.

\begin{pspicture}(-2,-2)(5,4)
\psline[showpoints=true, linestyle=dashed]%
(-1,2)(1,2)(2,3)
\psline[showpoints=true, linestyle=dashed]%
(-1,0)(1,0)(2,1)
\psline[showpoints=true, linestyle=dashed]%
(0,3)(0,1)
\psline[linestyle=dashed]%
(-1,0)(0,1)(2,1)
\psline[linestyle=dashed]%
(-1,2)(0,3)(2,3)
\psline[linestyle=dashed]%
(2,3)(2,1)
\psline[linestyle=dashed]%
(1,2)(1,0)
\psline[linestyle=dashed]%
(-1,0)(-1,2)
\psline[linestyle=dashed]%
(0,3)(1,0)
\psline[linestyle=dashed]%
(0,1)(1,2)
\psline[linestyle=dashed]%
(-1,2)(2,1)

\rput(0.2,-1){\rnode{A}{Pict.9}}
\rput(4.5,2){\rnode{B}{$k_{0}=1$}}
\rput(4.5,1.5){\rnode{C}{$1<{k_{1}}\leq{k_{2}}\leq{k_{3}}$}}

\rput(-0.3,3.1){\rnode{a}{$\widetilde{P_{1}}$}}
\rput(2.3,3.1){\rnode{b}{$\widetilde{P}$}}
\rput(0.7,2.3){\rnode{c}{$\widetilde{P_{2}}$}}
\rput(-1.3,2.2){\rnode{d}{$P_{3}$}}
\rput(-0.3,1.2){\rnode{e}{$P_{2}$}}
\rput(2.3,1.2){\rnode{f}{$\widetilde{P_{3}}$}}
\rput(1.3,-0.1){\rnode{g}{$P_{1}$}}
\rput(-1.3,-0.1){\rnode{h}{$P$}}

\end{pspicture}

 We see that only the conjugate vertices $P$ and $\widetilde{P}$ are not connected.  In addition to the faces $F_{l}$, $F_{l}'$, $l=1,2,\ldots,10$,there are the followig new faces:
\begin{equation*}
{F_{11}}=PP_{1}{\widetilde{P_{1}}}P_{3}, \quad {F_{12}}=PP_{2}{\widetilde{P_{1}}}P_{1}, \quad {F_{13}}=P_{2}{\widetilde{P_{3}}}P_{1}{\widetilde{P_{1}}}, \quad {F_{14}}=P_{1}{\widetilde{P_{1}}}{\widetilde{P_{3}}}P_{3}={F_{14}'}
\end{equation*}
and
\begin{equation*}
{F_{11}'}={\widetilde{P}}{\widetilde{P_{1}}}P_{1}{\widetilde{P_{3}}}, \quad {F_{12}'}={\widetilde{P}}{\widetilde{P_{2}}}P_{1}{\widetilde{P_{1}}}, \quad {F_{13}'}={\widetilde{P_{2}}}P_{3}{\widetilde{P_{1}}}P_{1}, \quad {F_{15}}={\widetilde{P_{1}}}P_{2}{\widetilde{P_{2}}}P_{1}={F_{15}'} 
\end{equation*}
Suppose that $k\in{\{10,14,15\}}$  and that $l\in{\{1,\ldots,9\}}\cup{\{11,12,13\}}$.     All faces $F_{l}$ and $F_{l}'$ have at most  one pair of conjugate points as their vertices. Hence, similarly as before, any  face $F_{l}$   determines the $3$-element set ${\widehat{Q}}^{F_{l}}\subseteq{\widehat{Q}}$ for each of its vertex $Q\in{\{F_{l}\}}$  (and the same is true for $F_{l}'$).

 On the other hand, when vertex $Q$ of a face $F_{k}$ is also a vertex of $F_{l}$ then    the bijections $\widehat{F}_{k}$ (given by the definition $7$), determine unique subset   ${{\widehat{F_{k}}}({\widehat{Q}}^{F_{l}})}\subset{\widehat{Q}}$  which is isomorphic to ${\widehat{Q}}^{F_{l}}\subset{\widehat{Q}}$.

  Proceeding in a similar way as before we will find that all admissible chains at a vertex $Q$ produce the spin group $\mathtt{G}_{Q}$ which is attached at this vertex.  More precisely we obtain:
 \begin{equation}
 \mathtt{G}_{Q}\cong
 \begin{cases}
 \mathtt{S}_{3}, \quad \text{when} \quad Q\in{\{P,\widetilde{P}\}}\\
 \mathtt{S}_{4}, \quad \text{when} \quad Q\in{\{P_{k},{\widetilde{P_{k}}}; k=1,2,3\}}
 \end{cases}
 \end{equation}

As before, spin chains $(4.5)$ at $Q$ which are not admissible will not move the elements of the set $\widehat{Q}\cong{\{1,2,3\}_{Q}}$, when $Q\in{\{P,{\widetilde{P}}\}}$, or of the set $\widehat{Q}\cong{\{1,2,3,4\}_{Q}}$ when $Q$ is any of the remaining vertices of $\mathcal{S}_{P}$.

\subsection{ For each $l=0,1,2,3$  we have $k_{l}>1$}
Now each pair of conjugate vertices of $\mathcal{S}_{P}$ is connected  and hence each vertex has $\epsilon$-degree equal to $4$. The connection graph is totally symmetric with respect to all of its vertices and for any $Q\in{\{\mathcal{S}_{P}\}}$ any spin chain $\mathsf{W}_{Q}$ at the vertex $Q$ is admissible. It is obvious (from the considerations above) that all spin chains at $Q$ will produce the spin group at $Q$ which is isomorphic to $\mathtt{S}_{4}$, i.e. $\mathtt{G}_{Q}={\mathtt{S}_{4}}(\widehat{Q})$.

All cases considered above show that the following lemma is true:
\begin{lemma}
Let $\mathcal{S}_{P}$ be an exceptional spin graph of degree $r=3$ on a hyperelliptic surface of genus $g\geq{4}$. The spin group $\mathtt{G}_{Q}$ attached to any of its vertex $Q$ is
\begin{equation*}
\mathtt{G}_{Q}\cong
\begin{cases}
\mathtt{S}_{3}, \quad \text{when} \quad {deg_{\epsilon}Q}=3=r \\
\mathtt{S}_{4}, \quad \text{when} \quad {deg_{\epsilon}Q}=4=r+1
\end{cases}
\end{equation*}
\end{lemma}

\section{SPIN GROUPS for GRAPHS od DEGREE $r>3$}
Let $\mathcal{S}_{P}$ be an exceptional spin graph of degree $r\geq{4}$ with a head $P$. Let $\mathcal{S}(r)$ be its basic graph and let ${\Gamma}\subset{\mathcal{S}(r)}\subset{\mathcal{S}_{P}}$ be any standard $3$-cell of the graph.  Since $\mathcal{S}_{P}$ is an exceptional graph it determines the unique 'decoration' ${\Gamma}^{\mathcal{S}}$ of this $3$-cell  $\Gamma$. More precisely, when  vertex $Q$ of $\Gamma$ has degree $deg_{\epsilon}Q=r+1$ then there is the additional edge $Q{\widetilde{Q}}$ in the graph ${\Gamma}^{\mathcal{S}}$.

Whenever the the image of $\Gamma^{\mathcal{S}}$ in the connection graph for $\mathcal{S}_{P}$ is not isomorphic to the standard spin-graph $\mathcal{S}(3)$ it will be called a decorated cell of the connection graph and it will be denoted as $\Gamma^{\mathcal{S}}$ as well.

 Let us notice that  we may have a situation that some of the decorated $3$-cells $\Gamma^{\mathcal{S}}$ (but not all) may be   isomorphic to the standard spin graph $\mathcal{S}(3)$  (i.e.  no pair of conjugate vertices of $\Gamma$ is connected in $\mathcal{S}_{P}$). So, any standard $3$-cell ${\Gamma}\subset{\mathcal{S}_{P}}$ corresponds to a unique decorating $3$-cell ${\Gamma}^{\mathcal{S}}$ which is isomorphic either to the standard spin graph $\mathcal{S}(3)$ on a surface of genus $3$ or it is isomorphic to an exactly one of the connection graphs  (of an exceptional spin graph of degree $r=3$) considered in the previous section.
 
 In a standard case any face $F$ determines a unique standard $3$-cell of the graph. However, when a graph is an exceptional one it is no longer a case.  More precisely, depending on a type of a face $F\subset{\mathcal{S}_{P}}$, we may distinguish the following possibilities:
 \begin{itemize}
 \item  The set $\{F\}$ of vertices does not contain  a pair of conjugate points (equivalently,  $F$ is a face of the basic graph $\mathcal{S}(r)$).  This implies that  $F$ determines a unique standard $3$-cell $\Gamma$ which may or may not be decorated.
 \item  $F$ is not a face of the basic graph $\mathcal{S}{r}\subset{\mathcal{S}_{P}}$ but it contains only one pair of mutually conjugate vertices. In this case, in the basic graph for $\mathcal{S}_{P}$,  there are $r-2$ distict standard $3$-cells   wich contain all vertices of $F$. Each of this cells determines a unique, necessarily decorated $3$-cell in the connection graph for $\mathcal{S}_{P}$.
 \item The vertices of $F$ consist of two pairs of conjugate points. Now the number of distict standard $3$-cells whose set of vertices contains $\{F\}$ is equal to $(r+1-2)(r+1-3)=(r-1)(r-2)$. Each of such cells determines a unique, decorated $3$-cell in the connection graph.
 \end{itemize}
The occurence of different types of  faces in the connection graph for $\mathcal{S}_{P}$ tells us that we should redefine the notion of a spin-chain $\mathsf{W}_{Q}$ at $Q\in{\{\mathcal{S}_{P}\}}$. 

So, let ${\mathsf{L}_{Q}}=QQ_{1}Q_{2}{\ldots}Q_{N}Q$ be any loop at $Q$ in an exceptional spin graph $\mathcal{S}_{P}$ of degree $r>3$.  Now, for each edge of this loop we must choose a standard $3$-cell ${\Gamma}\subset{\mathcal{S}(r)}$ and then a face of the unique (decorated) cell $\Gamma^{\mathcal{S}}$ along this edge. In other words we will define a spin chain $\mathsf{W}_{Q}$ at $Q$ as follows:
\begin{definition}
\begin{equation}
{\mathsf{W}_{Q}}=({{\mathsf{L}_{Q}}=QQ_{1}{\ldots}Q_{N}Q}; {\Gamma}_{1},\ldots,{\Gamma}_{N},{\Gamma}_{N+1}; F_{1},\ldots,F_{N},F_{N+1})
\end{equation}
\begin{equation*}
 {QQ_{1}}\subset{F_{1}}\subset{\Gamma_{1}}^{\mathcal{S}},\quad {Q_{l}Q_{l+1}}\subset{F_{l}}\subset{\Gamma_{l}}^{\mathcal{S}}, l=1,\ldots,N; \quad \text{and}\quad {Q_{N}Q}\subset{F_{N+1}}\subset{\Gamma_{N+1}^{\mathcal{S}}}
\end{equation*}
\end{definition}
Only when the set of verices $\{F_{l}\}$, ($l=1,\ldots,N+1$), contains no conjugate points (in this case we will say that face $F_{l}$ is a standard one) the face $F_{l}$ implies the unique  choice for  ${\Gamma}^{\mathcal{S}}_{l}$. Hence, when a loop $\mathsf{L}_{Q}$ lies in the basic graph ${\mathcal{S}(r)}\subset{\mathcal{S}_{P}}$ then, similarly as for $r=3$, we introduce basic chains
\begin{definition}
A basic chain at any vertex $Q\in{\{\mathcal{S}_{P}\}}$ is a chain 
\begin{equation*}
{\mathsf{W}_{Q}}=({\mathsf{L}_{Q}};F_{1},\ldots,F_{N+1})
\end{equation*}
 whose loop lies totally in the basic graph $\mathcal{S}(r)$  and whose all faces along the edges of its loop are also faces of the basic graph. 
\end{definition} 
 When a chain $\mathsf{W}_{Q}$ is a basic one then each face along an edge of its loop $\mathsf{L}_{Q}$ determines unique (decorated) $3$-cell and hence there is no need to indicate  cells as in the general definition $10$.

We recall that by choosing the  index  for the points $P_{1},\ldots,P_{r}$ of the set $\{\mathcal{A}_{P}\}$ we determine unique ordering of  the set $\widehat{Q}$ for any vertex $Q$ of $\mathcal{S}_{P}$. In other words, we may identify each set $\widehat{Q}$ either with $\{1,2,\ldots,r\}_{Q}$ when ${deg_{\epsilon}Q}=r$ or with the set $\{1,2,\ldots,r,r+1\}_{Q}$ when the degree of $Q$ is ${deg_{\epsilon}Q}=r+1$.

Let $F$ be a face of some decorated $3$-cell $\Gamma^{\mathcal{S}}$ of the connection graph for $\mathcal{S}_{P}$. Suppose that the set  ${\{F\}}={\{Q_{1},Q_{2},Q_{3},Q_{4}\}}$ of vertices of $F$ contains  at most one pair of conjugate points. This allows us 
  to  define the  $r$-element sets ${\widehat{Q_{l}}}^{\Gamma,F}$  for each $Q_{l}\in{\{F\}}$    as:
\begin{equation}
{\widehat{Q_{l}}}^{\Gamma,F}=
\begin{cases}
{\widehat{Q_{l}}} \quad \text{when}\quad deg_{\epsilon}Q_{l}=r\\
{\widehat{Q_{l}}}-{\frac{\widetilde{Q_{l}}}{Q_{l}}} \quad \text{when} \quad deg_{\epsilon}Q_{l}=r+1 \quad \text{and}\quad F\subset{\Gamma}\\
{\widehat{Q_{l}}}-{\frac{S}{Q_{l}}}\quad \text{when}\quad deg_{\epsilon}Q_{l}=r+1 \quad \text{and  $F$ is  not a standard face}
\end{cases}
\end{equation}
In the latter case point $S$ belongs to the set $\{S,{\widetilde{S}}\}=\{\Gamma\}-{\{{\{F\}}\cup{\widetilde{\{F\}}}\}}$.

 Since all sets $\widehat{Q_{l}}$ are ordered we will introduce the following identifications:
\begin{equation}
{\widehat{Q_{l}}}^{\Gamma,F}\cong
\begin{cases}
{\{1,2,\ldots,r\}}_{Q_{l}} \quad \text{when} \quad deg_{\epsilon}Q_{l}=r\\
{\{j_{1},j_{2},\ldots,j_{r}\}}\subset{\{1,2,\ldots,r,r+1\}_{Q_{l}}} \quad \text{otherwise}
\end{cases}
\end{equation}

First  we will use our decorated cell $\Gamma^{\mathcal{S}}$ and its face $F$  to find one-to-one correspondences between the appropriate $3$-element subsets ${\widehat{Q_{l}}}^{\Gamma}$ of ${\widehat{Q_{l}}}^{\Gamma,F}$ for $l=1,..,4$.   We will denote these subsets  by ${\widehat{Q_{l}}}^{\Gamma}$ since they are totally determined only by the cell $\Gamma^{\mathcal{S}}$.  We have  
\begin{equation}
{\widehat{Q_{l}}}^{\Gamma}\cong
\begin{cases}
{\{i_{1},i_{2},i_{3}\}}\subset{\{1,2,\ldots,r\}_{Q_{l}}}\quad \text{when}\quad deg_{\epsilon}Q_{l}=r\\
{\{i_{1},i_{2},i_{3},i_{4}\}}\subset{\{1,2,\ldots,r,r+1\}_{Q_{l}}} \quad \text{otherwise}
\end{cases}
\end{equation}
Besides, we intoduce the following subsets: 
\begin{equation}
{\widehat{\widehat{Q_{l}}}^{\Gamma}}={{{\widehat{Q_{l}}}^{\Gamma}}\cap{\widehat{Q_{l}}}^{\Gamma,F}}\cong{\{{\overline{1}},{\overline{2}},{\overline{3}}\}^{\Gamma}_{Q_{l}}}; \quad l=1,\ldots,4
\end{equation}

    Since ${\Gamma}^{\mathcal{S}}$ is isomorphic either to the standard graph $\mathcal{S}(3)$ or to the connection graph for some exceptional spin  graph of order $3$, we use [3] or we use the section $4$ above to find an isomorphism      ${\widehat{\widehat{F}}^{\Gamma}_{Q_{l},Q_{j}}}: {\widehat{\widehat{Q}}^{\Gamma}_{l}}\rightarrow{\widehat{\widehat{Q}}^{\Gamma}_{j}}$. According to our notation, this isomorphism can be represented as a permutation 
\begin{equation*}
\sigma^{\Gamma,F}_{Q_{l},Q_{j}}: {\{{\overline{1}},{\overline{2}},{\overline{3}}\}^{\Gamma}_{Q_{l}}}\rightarrow{\{{\overline{1}},{\overline{2}},{\overline{3}}\}^{\Gamma}_{Q_{j}}} 
\end{equation*}

 To find correspondences between the   remaining elements of ${\widehat{Q_{l}}}^{\Gamma,F}$ we proceed as follows: For each ${R}\notin{\{\Gamma\}}$, we will identify $\frac{R}{Q_{l}}$ (or $\frac{\widetilde{R}}{Q_{l}}$, depending which one belongs to $\widehat{Q_{l}}$) with ${\frac{R}{Q_{j}}}$   ( or with $\frac{\widetilde{R}}{Q_{j}}$ depending which one is an element of ${\widehat{Q_{j}}}$). Here  $l,j\in{\{1,2,3,4\}}$ and $\frac{R}{Q_{l}}$ is the shorthand for ${\Phi_{Q_{l}}}(R)\in{\widehat{Q_{l}}}\subset{\Theta_{\epsilon}}$. 
 
 In this way we have constructed bijections  ${\widehat{F}}^{\Gamma}$,  $F\subset{\Gamma}^{\mathcal{S}}$  between all $r$-element sets corresponding to the  vertices of $F$, i.e.
\begin{equation}
{\widehat{F}}^{\Gamma}:  {\widehat{Q_{1}}}^{\Gamma,F}\leftrightarrow{\widehat{Q_{2}}}^{\Gamma,F}\leftrightarrow{\widehat{Q_{3}}}^{\Gamma,F}\leftrightarrow{\widehat{Q_{4}}^{\Gamma,F}}
\end{equation}
 Each isomorphism ${\widehat{\widehat{F}}}^{\Gamma}_{Q_{l},Q_{j}}$ can be obtained by  the restriction of ${\widehat{F}}^{\Gamma}$ to the appropriate subsets.  Moreover, when $deg_{\epsilon}Q_{l}=deg_{\epsilon}Q_{j}=r+1$ for some $l,j\in{\{1,2,3,4\}}$ then ${F}^{\Gamma}$ has natural expansion to the bijection $F^{\Gamma}: {\widehat{Q_{l}}}\leftrightarrow{\widehat{Q_{j}}}$.

Now let us suppose that the set $\{F\}={\{Q_{1},\ldots,Q_{4}\}}$ of verices of $F$ consists of two pairs of conjugate points. This means that each $Q_{l}$ has the degree equal to $r+1$.  In such  case,  the bijections between the $(r+1)$-element sets ${\widehat{Q_{l}}}$, ($l=1,..,4$),  do not depend on a choice of a $3$-cell $\Gamma$ with $F\subset{\Gamma}^{\mathcal{S}}$.  In other words,  for any decorated cell ${\Gamma}^{\mathcal{S}}$ with a face $F$   the produced mappings will be exactly the same. More precisely,  for each $R\notin{\{F\}}$     these bijections 
   identify the elements $\frac{R}{Q_{l}}$ (or ${\frac{\widetilde{R}}{Q_{l}}}$, depending which one belongs to ${\widehat{Q_{l}}}$) for $l=1,2,3,4$ between themselves.   The correspondences between the remaining two elements of each set ${\widehat{Q_{l}}}$ are determined by $F$ in the obvious way. Summarizing, our face $F$ defines natural bijections between $(r+1)$-element sets:
       \begin{equation}
       {\widehat{F}}:  {\widehat{Q_{1}}}\leftrightarrow{\widehat{Q_{2}}}\leftrightarrow{\widehat{Q_{3}}}\leftrightarrow{\widehat{Q_{4}}}
       \end{equation} 
 Now, let $Q$ be any vertex of ${\mathcal{S}_{P}}$ and   let ${\mathsf{W}_{Q}}$ be a chain at $Q$ as defined in $(5.1)$. Suppose that 
 $R_{1}R_{2}$ is an edge of the loop ${\mathsf{L}_{Q}}$ and  suppose that $deg_{\epsilon}R_{1}=deg_{\epsilon}R_{2}=r<g$. This means that 
 for any standard cell $\Gamma$ with an edge $R_{1}R_{2}$ and for any face $F\subset{\Gamma}^{\mathcal{S}}$ along this edge, $F$
  has at most one pair of conjugate points as its vertices. Moreover, we have 
    ${\widehat{R_{k}}}^{\Gamma,F}={\widehat{R_{k}}}\cong{\{1,2,\ldots,r\}_{R_{k}}}$ for $k=1.2$.

             The bijections ${\widehat{F}}^{\Gamma}$  (see $(5.6)$) determine the isomorphism (denoted in the same way)  ${\widehat{F}}^{\Gamma}: {\widehat{R_{1}}}\rightarrow{\widehat{R_{2}}}$ which   can be represented by a unique  permutation
\begin{equation*}
{\sigma}^{F}_{R_{1},R_{2}}: {\{1,2,\ldots,r\}_{R_{1}}}\rightarrow{\{1,2,\ldots,r\}_{R_{2}}}
\end{equation*}  

Now suppose that ${deg_{\epsilon}R_{1}}=r$ and ${deg_{\epsilon}R_{2}}=r+1$.  As in the latter case, this also implies that each face $F$ with an edge $R_{1}R_{2}$ has  at most one pair of conjugate points   in the set of its vertices.   Let  ${\Gamma}\subset{\mathcal{S}(r)}$ be a standard $3$-cell  such that ${F}\subset{\Gamma}^{\mathcal{S}}$.  Since ${\widehat{R}}^{\Gamma}_{2}$ must contain the element $\frac{\widetilde{R_{2}}}{R_{2}}$  we have
\begin{equation*}
{\widehat{R}}^{\Gamma}_{1}\cong{\{{\overline{1}},{\overline{2}},{\overline{3}}\}^{\Gamma}_{R_{1}}}\quad \text{and} \quad {\widehat{R}}^{\Gamma}_{2}\cong{\{{\overline{1}},{\overline{2}},{\overline{3}},{\overline{4}}\}^{\Gamma}_{R_{2}}}
\end{equation*}
as well as the well defined isomorphism
\begin{equation}
{\widehat{\widehat{F}}}^{\Gamma}_{R_{1},R_{2}}: {\widehat{R}}^{\Gamma}_{1}\rightarrow{\widehat{\widehat{R}}}^{\Gamma,F}_{2}
\end{equation}
Here ${\widehat{\widehat{R}}}^{\Gamma,F}_{2}$ is the $3$-element subset  of the set ${\widehat{R}}^{\Gamma}_{2}$ as defined by $(5.5)$ ( i.e. as the intersection  ${\widehat{R_{2}}}^{\Gamma}\cap{\widehat{R_{2}}}^{\Gamma,F}$). Of course, by $(5.6)$ this isomorphism ${\widehat{\widehat{F}}}^{\Gamma}_{R_{1},R_{2}}$ has the natural expansion to the isomorphism $F^{\Gamma}:{\widehat{R_{1}}}\rightarrow{\widehat{R_{2}}}^{\Gamma,F}$.

When the degrees of both vertices $R_{1}$ and $R_{2}$ are equal to r+1 then for any face $F$ with an edge $R_{1}R_{2}$ and for any standard $3$-cell  ${\Gamma}\subset{\mathcal{S}}(r)$ with $F\subset{\Gamma}^{\mathcal{S}}$ we have ${\widehat{R}^{\Gamma}_{k}}\cong{\{{\overline{1}},{\overline{2}},{\overline{3}},{\overline{4}}\}^{\Gamma}_{R_{k}}}$,  $k=1,2$. Since ${\Gamma}^{\mathcal{S}}$ is isomorphic to the connection graph of some exceptional  spin graph of order $3$, we easily obtain (using  the section $4$) the isomorphism ${\widehat{F}}^{\Gamma}_{R_{1},R_{2}}: {{\widehat{R}}^{\Gamma}_{1}}\rightarrow{{\widehat{R}}^{\Gamma}_{2}}$ and then its natural expansion 
\begin{equation}
{\widehat{F}}: {\widehat{R_{1}}}\rightarrow{\widehat{R_{2}}}
\end{equation}

Moreover, when our face $F$ has at most one pair of conjugate points as its vertices then $F$ defines the $3$-element subsets ${\widehat{\widehat{R}}}^{\Gamma,F}_{k}$ of ${\widehat{R}}^{\Gamma}_{k}\subset{\widehat{R_{k}}}$  as well as the $r$-elements subsets ${\widehat{R_{k}}}^{\Gamma,F}$ of the sets $\widehat{R_{k}}$  respectively, $k=1,2$. Obviously, the restrictions of the isomorphism $(5.9)$ to the appropriate subsets produce the isomorphism between these, introduced earlier,  subsets.

\begin{definition}\hfill
Let  us consider  a spin chain ${\mathsf{W}_{Q}}$  at $Q$  of the form given by $(5.1)$.
\begin{itemize}
\item When all vertces of the loop $\mathsf{L}_{Q}$ have the same degree (equal either to $r$ or to $r+1$ respectively) then all spin chains $(5.1)$ are admissible.
\item When  $\mathsf{L}_{Q}$ has vertices with different   $\epsilon$-degrees  then $\mathsf{W}_{Q}$ is admissible if all compositions of isomorphisms ${\widehat{F_{l}}}^{\Gamma_{l}}$ restricted to appropriate $r$-element sets for $l=1,\ldots,N+1$, are well defined 
\end{itemize}
\end{definition}

To illustrate the situation we will consider an exceptional spin graph $\mathcal{S}_{P}$ of order $r=4<g$ with a head $P$. The basic graph $\mathcal{S}(4)$ for $\mathcal{S}_{P}$ has $5$ different $3$-cells and hence there are $5$ decorated $3$-cells $\Gamma^{\mathcal{S}}_{k}\subset{\mathcal{S}_{P}}$. Let $\Gamma_{k}$ denote a cell with the vertices ${\{\Gamma_{k}\}}={\{\mathcal{S}_{P}\}}-{\{P_{k},{\widetilde{P_{k}}}\}}$ for $k=1,2,3,4$ and let the set of vertices of the cell $\Gamma_{5}$ be ${\{\Gamma_{5}\}}={\{P_{j},{\widetilde{P_{j}}}; j=1,\ldots,4\}}$.

$\mathbf{EXAMPLE}$ $\mathbf{I}$. Suppose that in $\mathcal{S}_{P}$ only one pair, $P_{4},\widetilde{P_{4}}$ of conjugate vertices is connected. In this case all cells $\Gamma^{\mathcal{S}}_{k}$  for $k\neq{4}$ are decorated but the cell $\Gamma^{\mathcal{S}}_{4}$ is isomorphic to the standard spin graph on a surface of genus $3$. Let us consider the following  chain at $P$:
\begin{equation}
{\mathsf{W}}_{P}: P\stackrel{\Gamma_{1},F_{1}}{\rightarrow}P_{4}\stackrel{\Gamma_{3},F_{2}}{\rightarrow}{\widetilde{P_{4}}}\stackrel{\Gamma_{5},F_{3}}{\rightarrow}P_{2}\stackrel{\Gamma_{1},F_{4}}{\rightarrow}P
\end{equation}
with
\begin{equation*}
F_{1}=PP_{2}{\widetilde{P_{4}}}P_{4}\subset{\Gamma_{1}},\quad F_{2}=PP_{1}{\widetilde{P_{4}}}P_{4}\subset{\Gamma_{3}}, \quad F_{3}={\widetilde{P_{1}}}P_{2}{\widetilde{P_{4}}}P_{4}\subset{\Gamma_{5}},\quad F_{4}=F_{1}
\end{equation*}
Since all compositions of isomorphisms between appropriate $4$ element sets are well defined, this chain is admissible.  It determines the isomorphism of the set ${\widehat{P}}\cong{\{1,2,3,4\}_{P}}$ which can be represented by the permutation ${\sigma^{\mathsf{W}_{P}}}=(143)\in{\mathtt{S}_{4}}(\widehat{P})$. Another chain at $P$ given as 
\begin{equation}
{\mathsf{W}}_{P}; P\stackrel{\Gamma_{3},F_{1}}{\rightarrow}P_{1}\stackrel{\Gamma_{2},F_{2}}{\rightarrow}P
\end{equation}
with $F_{1}=PP_{1}{\widetilde{P_{4}}}P_{4}\subset{\Gamma_{3}}$ and with $F_{2}=PP_{1}{\widetilde{P_{3}}}P_{4}\subset{\Gamma_{2}}$ is admissible as well and produces the isomorphism of ${\widehat{P}}$ represented by the permutation ${\sigma^{\mathsf{W}_{P}}}=(13)\in{\mathtt{S}_{4}}(\widehat{P})$.

On the other hand the chain
\begin{equation}
{\mathsf{W}_{P}}: P\stackrel{\Gamma_{1},F_{1}}{\rightarrow}P_{4}\stackrel{\Gamma_{2},F_{2}}{\rightarrow}{\widetilde{P_{3}}}\stackrel{\Gamma_{5},F_{3}}{\rightarrow}P_{1}\stackrel{\Gamma_{3},F_{4}}{\rightarrow}P
\end{equation}
with $F_{1}=PP_{2}{\widetilde{P_{4}}}P_{4}\subset{\Gamma_{1}}$, $F_{2}=P_{4}{\widetilde{P_{3}}}{\widetilde{P}}{\widetilde{P_{1}}}\subset{\Gamma_{2}}$, $F_{3}={\widetilde{P_{3}}}P_{1}P_{2}{\widetilde{P_{4}}}\subset{\Gamma_{5}}$ and with $F_{4}=PP_{1}{\widetilde{P_{4}}}P_{4}\subset{\Gamma_{3}}$ is not an admissible chain at $P$. The reason for this is that the composition ${\widehat{F_{2}}}^{\Gamma_{2}}\circ{\widehat{F_{1}}}^{\Gamma_{1}}$ is not defined. We assume that any such not admissible chain at a vertex $Q$ does not move the elements of $\widehat{Q}$ at all.

$\mathbf{EXAMPLE}$ $\mathbf{II}$. Assume that   spin graph $\mathcal{S}_{P}$ has two pairs of vertices, $\{P_{4},{\widetilde{P_{4}}}\}$ and $\{P_{3},{\widetilde{P_{3}}}\}$  connected. This implies that all $3$-cells are decorated.   The $\epsilon$-degrees of $P$,$P_{1}$ and $P_{2}$ are equal to $4$ but the degrees of $P_{k}$ for $k=3,4$ are equal to $5$. Now, ${\Gamma_{3}}^{\mathcal{S}}$ and ${\Gamma_{4}}^{\mathcal{S}}$ are isomorphic to the standard cell with only one additional edge (connecting appropriate conjugate points), whereas the remaining decorated cells of the connection graph for $\mathcal{S}_{P}$, have two additional edges each. Let us consider a chain at $P_{2}$ given by
\begin{equation}
{\mathsf{W}}_{P_{2}}:  P_{2}\stackrel{\Gamma_{4},F_{1}}{\rightarrow}{\widetilde{P_{3}}}\stackrel{\Gamma_{2},F_{2}}{\rightarrow}P_{4}\stackrel{\Gamma_{5},F_{3}}{\rightarrow}{\widetilde{P_{3}}}\stackrel{\Gamma_{5},F_{4}}{\rightarrow}P_{2}
\end{equation}
 where $F_{1}=PP_{2}{\widetilde{P_{3}}}P_{3}\subset{\Gamma_{4}}$, $F_{2}=P_{3}{\widetilde{P_{4}}}P_{4}{\widetilde{P_{3}}}\subset{\Gamma_{2}}$,  $F_{3}=P_{4}{\widetilde{P_{3}}}P_{3}{\widetilde{P_{2}}}\subset{\Gamma_{5}}$ and $F_{4}={\widetilde{P_{3}}}P_{2}{\widetilde{P_{4}}}P_{3}\subset{\Gamma_{5}}$. Since we have 
 \begin{equation*} {\widehat{\widetilde{P_{3}}}}^{F_{3}}={\widehat{F_{3}}}^{\Gamma_{5}}\circ{\widehat{F_{2}}}^{\Gamma_{2}}\circ{\widehat{F_{1}}}^{\Gamma_{4}}({\widehat{P_{2}}})={\widehat{\widetilde{P_{3}}}}^{F_{4}}
\end{equation*}
this chain is admissible and it produces an isomorphism of $\widehat{P_{2}}$ corresponding to the odd permutation $\sigma^{\mathsf{W}_{P_{2}}}={(12)}\in{\mathtt{S}_{4}}(\widehat{P_{2}})$. 
\begin{lemma}
Let $\mathcal{S}_{P}$ be an exceptional spin graph of degree $r\geq{4}$ with a head $P$. Let $Q$ be any vertex of $\mathcal{S}_{P}$. The set of all admissible chains at $Q$ produces the spin group $\mathtt{G}_{Q}$ at $Q$ which is either isomorphic to the symmetry group $\mathtt{S}_{r}$, when $deg_{\epsilon}Q=r$ or it is isomorphic to the group $\mathtt{S}_{r+1}$, when $deg_{\epsilon}Q=r+1$.
\end{lemma} 
\begin{proof}
It is obvious that the composition of any two admissible spin chains at a vertex $Q\in{\{\mathcal{S}_{P}\}}$ is again an admissible chain at $Q$. Moreover, for any admissible spin chain $\mathsf{W}_{Q}$  the traveling of its loop $\mathsf{L}_{Q}$ in  two opposite directions produce   permutations of the set $\widehat{Q}$ which are inverse to each other. This means that the set of all isomorphisms of the set $\widehat{Q}$ produced by the set of all admissible chains at $Q$ form a group which  will be  denoted  by $\mathtt{G}_{Q}$.

We already know that the set of all basic chains at $Q$ produces a subgroup of $\mathtt{G}_{Q}$ which is isomorphic to the alternating group $\mathtt{A}_{r}$ acting on the set ${\widehat{Q}}$ when $deg_{\epsilon}Q=r$ or acting on the set ${\widehat{Q}}-{\{\frac{\widetilde{Q}}{Q}\}}$ when $deg_{\epsilon}Q={r+1}$.

  Now let us suppose that  $\mathsf{W}_{Q}$ is any admissible chain at a vertex $Q$  whose loop $\mathsf{L_{Q}}$   lies totally in  some decorated cell $\Gamma^{\mathcal{S}}$. Suppose that  all faces along the edges of $\mathsf{L}_{Q}$  are also in $\Gamma^{\mathcal{S}}$.  We already know (from the section $(4)$)    that the chain $\mathsf{W}_{Q_{k}}$ produces a permutation $\sigma^{\mathsf{W}_{Q}}$. More precisely, when $deg_{\epsilon}Q=r$ then this permutation belongs to the symmetry group $\mathtt{S}_{3}$ acting on the set $\{{\overline{1}},{\overline{2}},{\overline{3}}\}^{\Gamma}_{Q}$ and   can be obviously extended to (denoted in the same way) 
  \begin{equation}
  {\sigma^{\mathsf{W}_{Q}}: {\{1,2,\ldots,r\}}_{Q}\rightarrow{\{1,2,\ldots,r\}}_{Q}}
  \end{equation}
When the degree of $Q$ is equal to $r+1$ then ${\sigma}^{\mathsf{W}_{Q}}$ belongs to the symmetry group $\mathtt{S}_{4}$ acting on the set $\{{\overline{1}},{\overline{2}},{\overline{3}},{\overline{4}}\}^{\Gamma}_{Q}$ and extends uniquely  to the permutation
\begin{equation}
\sigma^{\mathsf{W}_{Q}}:{\{1,2,\ldots,r+1\}_{Q}}\rightarrow{\{1,2,\ldots,r+1\}_{Q}}
\end{equation}
Suppose that the degree of a vertex $Q$ is equal to $r$. Since any vertex of an exceptional spin graph $\mathcal{S}_{P}$ must  also be  a vertex of some decorated cell ${\Gamma^{\mathcal{S}}}\neq{\Gamma}$  we will have some odd permutation $(5.14)$ that is produced by an appropriate chain $\mathsf{W}_{Q}$ in ${\Gamma}^{\mathcal{S}}$.  Since each odd permutation in $\mathtt{S}_{r}$ can be obtained as a product of any fixed odd permutation $\sigma\in{\mathtt{S}_{r}}$  and some element of the alternating subgroup ${\mathtt{A}_{r}}\triangleleft{\mathtt{S}_{r}}$, we obtain immediately that the spin group at $Q$ is ${\mathtt{G}_{Q}}\cong{\mathtt{S}_{r}}$.

Suppose that $deg_{\epsilon}Q={r+1}$. In this case each $3$-cell with a vertex $Q$ must be decorated i.e. each ${\Gamma^{\mathcal{S}}}\neq{\Gamma}$. Since any of such cells produce even and  odd permutations $(5.15)$ it is easy to see that these permutation together with $\mathtt{A}_{r}$ (acting on the r-element set ${\widehat{Q}}-{\{\frac{\widetilde{Q}}{Q}\}}$) suply all generators of the full symmetry group $\mathtt{S}_{r+1}$ acting on the set $\widehat{Q}$.  This completes our proof.
 \end{proof} 
 
 \section{SUMMARY}

 Any even nonsingular spin bundle $\xi_{\epsilon}$ on a hyperelliptic Riemann surface $\Sigma$ determines a foliation of $\Sigma$. All leaves of this foliation are finite and any leaf carry an additional structure of a spin graph. Almost all leaves have $2g+2$ points that form vertices of standard spin graph which has to belong to the unique isomorphism class $\mathcal{S}(g)$ of graphs  depending only on the genus $g$ of a surface.  Besides there are two leaves  through the Weierstrass points.  Each of them has $g+1$ points that form vertices of a  Weierstrass spin graph.  At every point of a standard leaf and at each point of a Weierstrass leaf the spin groups are isomorphic to each other and they all are isomorphic to the alternating group $\mathtt{A}_{g}$. At exceptional poins these groups are different. They number as well as they isomorhic classes are given by the presence of a concrete number of exceptional spin graphs in any possible class of isomorphic graphs.  

 The exceptional spin groups attached to the vertices of an exceptional spin graph, say $\mathcal{S}_{P}$, are determined only by the connection graph for $\mathcal{S}_{P}$  and  we have seen that   non-isomorphic exceptional graphs may have exactly the same connection graph. Hence, the classification of hyperelliptic Riemann surfaces using exceptional spin groups  is different that the classification determined only by the exceptional spin graphs.
 
 Since any surface  $\Sigma$  must carry exceptional spin graphs ~\cite{KB13}, we must have some finite number of exceptional spin groups  (which have been found in this paper)  attached to exceptional points of $\Sigma$.
  These exceptional groups are obstacles for construction of a principal bundle over $\Sigma$ with the strucrural group given by $\mathtt{A}_{g}$. In other words, on any hyperelliptic Riemann surface equipped with a nonsingular even spin structure, 
    we have some sort of a  'singular principal' $\mathtt{A}_{g}$-bundle.


\begin{thebibliography}{8}

\bibitem{RM95}Miranda,R.,
\emph{Algebraic curves and Riemann Surfaces},
AMS, GSM vol.5, 1995


\bibitem{DV11}Varolin,D.,
\emph{Riemann Surfaces by Way of Complex Analytic Gepmetry},
AMS, GSM vol.125, 2011


\bibitem{RG67}Gunning,R.C.,
\emph{Lecture on vector bundles over Riemann surfaces},
MNPUP, Princeton, 1967

\bibitem{KB13}Bugajska,K.,
\emph{Spin Graphs},
submitted for publication

\bibitem{KMB13}Bugajska,K.,
\emph{Standard and Weierstrass spin groups on hyperelliptic Riemann surfaces},
submitted for publication

\bibitem{FK92}Farkas,H.M.,~Kra,I.,
\emph{Riemann Surfaces}
Springer-Verlag, GTM 71, 1992

\bibitem{FK01}Farkas,H.M.,~Kra,I.,
\emph{Theta constants, Riemann surfaces and the Modular Group},
AMS, GTM vol.37, 2001 


\end{thebibliography}
\end{document}